\documentclass{amsart}
\usepackage{amsmath,amsthm,amsfonts,graphicx,amssymb}
\usepackage{amscd}
\usepackage{mathrsfs}
\usepackage[utf8]{inputenc}
\usepackage[left=30mm,right=30mm]{geometry}
\usepackage{mathtools}
\usepackage{tikz-cd}
\usepackage{enumitem}
\usepackage{quiver}
\usepackage{thm-restate}

\usepackage{hyperref}
 \usepackage{dirtytalk}

\def\dist{{d}}
\def\RR{{\mathbb R}}

\long\def\comment#1\endcomment{}
\usepackage{color}

\title[Weak Morse and combings]{Weak Morse property in spaces with bounded combings}
\author{Cornelia Dru\c{t}u}
\email{drutu@maths.ox.ac.uk}
\address{Mathematical Institute \\
 University of Oxford \\ Oxford, UK.}
\author{Davide Spriano}
\email{spriano@maths.ox.ac.uk}
\address{Mathematical Institute \\
 University of Oxford \\ Oxford, UK.}
 \author{Stefanie Zbinden}
\email{sz2020@hw.ac.uk}
\address{Heriot-Watt University}
\date{\today}

\definecolor{darkgreen}{cmyk}{1,0,1,.2}
\allowdisplaybreaks[1] 

\long\def\comment#1\endcomment{}

\numberwithin{equation}{section}
\newtheorem{theorem}[equation]{Theorem}

\newtheorem{corollary}[equation]{Corollary}
\newtheorem{lemma}[equation]{Lemma}

\newtheorem{notation}[equation]{Notation}
\newtheorem{cvn}[equation]{Convention}

\newtheorem{question}[equation]{Question}

\newtheorem{definition}[equation]{Definition}

\newtheorem{claim}{Claim}
\theoremstyle{definition}
\newtheorem{example}[equation]{Example}
\newtheorem{remark}[equation]{Remark}

\newtheoremstyle{citing}
  {3pt}
  {3pt}
  {\itshape}
  {}
  {\bfseries}
  {}
  {.5em}
  {\thmnote{#3}}

\theoremstyle{citing}

\DeclarePairedDelimiter\abs{\lvert}{\rvert}%
\DeclarePairedDelimiter\norm{\lVert}{\rVert}
\newcommand{\eps}{\epsilon}

\newcommand{\mb}{\partial_*}

\newcommand {\la}{\lambda}
\newcommand{\lao}{\lambda_0}
\newcommand {\ka}{\kappa}
\newcommand{\kao}{\kappa_0}
\newcommand{\net}{{n(\eta)}} 
\newcommand{\nto}{{n}} 
\newcommand{\exit}{{t_e}}
\newcommand{\gr}{\gamma_R} 
\newcommand{\bt}{\mathbf{t}} 
\newcommand{\nn}{{\mathcal N}}
\newcommand{\bp}{e}

\newcommand{\wmorse}{\mu}
\newcommand{\col}{\qgot}
\newcommand{\W}{\rho}
\renewcommand{\H}{\tau}
\newcommand{\K}{\chi}

\newcommand{\C}{C}

\newcommand{\length}[1]{{\norm{{#1}}}}
\newcommand{\dlength}[1]{{\abs{{#1}}}}

\newcommand{\R}{\mathbb{R}}

\newcommand{\N}{\mathbb{N}}

\newcommand {\qgot}{{\mathfrak q}}
\newcommand {\pgot}{{\mathfrak p}}

\newcommand {\tgot}{{\mathfrak t}}

\newcommand {\mf}[1]{\mathfrak{#1}}
\newcommand{\mc}{\mathcal}

\newcommand{\deletethis}[1]{}

\numberwithin{equation}{section}

\date{}
\begin{document}

\begin{abstract}
We relate two notions of non-positive curvature: bounded combings and the Morse local-to-global (MLTG) property (in its weak and strong version). The latter is a property of a space that has been shown to eliminate pathological behavior of Morse geodesics. We showcase its importance in a survey in the appendix. We show that having a bounded combing implies the weak MLTG property. If the Morse boundary of a group is $\sigma$--compact, we show that the weak MLTG property is upgraded to the (strong) MLTG property.
\end{abstract}

\maketitle

\section{Introduction}

A \emph{Morse geodesic} is, roughly speaking, a geodesic near which the space has negative curvature features. Thus, every geodesic in a Gromov hyperbolic space is Morse and, conversely, if every geodesic in a space is uniformly Morse, then the space is hyperbolic \cite{Bonk:quasigeodesic}. It is therefore not surprising that, in the study of spaces of coarse non-positive curvature, the investigation of the behaviour of Morse geodesics has played an important part \cite{OlshankiiOsinSapir:lacunary,DurhamTaylor:convexcocompact, Cordes:Morse,aougabdurhamtaylor:pulling,CharneyCordesSisto:complete,zbinden:morse, fioravantikarrersistozbinde:cech}. A key object in this investigation is a particular type of boundary associated to Morse geodesics, the \emph{Morse boundary}. Introduced in \cite{Cordes:Morse} building upon \cite{CharneySultan:contracting}, the Morse boundary is a topological space that \say{encodes Morse directions}, in a way analogous to the Gromov boundary.

Unfortunately, Morse geodesics do not behave quite as well as geodesics in hyperbolic spaces, obstructing several results from being generalized. For instance, a (non-virtually cyclic) group may contain a Morse quasi-geodesic (with respect to a word metric) without containing a Morse element \cite{fink:morse}, or it may contain a Morse element, but not a free subgroup \cite{OlshankiiOsinSapir:lacunary}. However, if a group satisfies the \emph{Morse-local-to global (MLTG)} property, then these pathologies disappear \cite{RussellSprianoTran:thelocal} and a wide range of geometric, algebraic, topological and language-theoretical properties are satisfied \cite{RussellSprianoTran:thelocal, CordesRussellSprianoZalloum:regularity, sistozalloum:morse, minehspriano:separability, IslamWeisman:morseproperties, HeSprianoZbinden:sigmacompactnes}. We will briefly overview these properties at the end of the introduction. 

Thus, it is very useful to identify classes of spaces that satisfy the MLTG property.

\subsection{Main results}

In our first result, we do this for certain spaces with bounded combings. See Section \ref{sec:combings} for a definition of bounded combing.

\begin{theorem}\label{thm:main_intro}
 Let $X$ be a geodesic metric space on which $\mathrm{Isom}(X)$ acts coboundedly. If $X$ has a bounded quasi-geodesic combing and  $\sigma$--compact Morse boundary then $X$ satisfies the MLTG property.
\end{theorem}

Theorem~\ref{thm:main_intro} is the combination of three lines of inquiry, each of independent interest, concerned respectively with the globalisation of quasi-geodesics (Theorem~\ref{introthm:globalization_qg:v2}), the weak Morse local-to-global property (Theorem~\ref{thmi:WMLTG}), and the promotion of the weak MLTG property to the (strong) MLTG property (Theorem~\ref{thmi:non-sigma}). 

Globalisation of quasi-geodesics means understanding when a local quasi-geodesic is a global quasi-geodesic. In \cite[Proposition~7.2.E]{Gromov(1987)}, Gromov proved that if, in a metric space, all paths which are local quasi-geodesics are global quasi-geodesics (choosing quantifiers appropriately), then the space is hyperbolic. In particular, for non-hyperbolic spaces, we need to impose extra conditions on local quasi-geodesics to have any hope of them being global quasi-geodesics. 

The extra condition we consider is a very natural notion of convexity called the \emph{weak Morse property} (Definition~\ref{defn:weak_Morse}). A quasi-geodesic is $(Q,q, \wmorse)$--\emph{weakly Morse} if $(Q,q)$--quasi-geodesics with endpoints on it stay in its $\wmorse$--neighbourhood (in other words, if it is $\wmorse$--quasi-convex with respect to $(Q,q)$--quasi-geodesics). We show that in a geodesic space with a combing, paths that locally are weakly Morse quasi-geodesics are global quasi-geodesics.

\begin{theorem}\label{introthm:globalization_qg:v2}
    Let $X$ be a geodesic metric space with a bounded $(\lao, \kao)$--quasi-geodesic combing. For every quasi-geodesic pair of constants $(\la, \ka)$ and every $\wmorse \geq 0$, there exists a scale $D\geq 0$, and quasi-geodesic constants $(\la', \ka')$ such that every path $\pgot$ that $D$--locally is a $(2\lao +2, \kao, \mu)$--weakly Morse $(\la,\ka)$--quasi-geodesic is a global $(\la',\ka')$--quasi-geodesic. 
\end{theorem}
A notable feature of Theorem~\ref{introthm:globalization_qg:v2} is that it can be made effective and implemented in computer computations. Indeed, if a graph is transitive (say, a Cayley graph), it is possible to list all weakly Morse geodesics of length at most $D$, and then to verify whether a path is a quasi-geodesic by checking if every subsegment of the path is on the list. 

The assumptions of Theorem~\ref{introthm:globalization_qg:v2} cannot be too far from optimal, as the following example shows. Let $X = \vee_{n=1}^\infty (S^1, nd)$ be the wedge sum of (pointed) circles of radius $n$ equipped with the path metric. Then the above theorem fails in $X$ even under the stronger assumption that the path $\pgot$ is locally a \emph{Morse} quasi-geodesic, that is, a quasi-geodesic which is weakly Morse with respect to all quasi-geodesic constants $(Q,q)$ (see Definition~\ref{defn:Morse_QG}). 

We continue the study of globalisation properties by investigating when the resulting quasi-geodesic is weakly Morse. In particular, we adopt the formalism of the established Morse local-to-global property and say that a space satisfies the \emph{weak Morse local-to-global property} (weak MLTG) if every path that locally is a Morse quasi-geodesic is globally a weakly Morse quasi-geodesic. We refer the reader to Definition~\ref{defn:wmltg} for the statement with precise quantifiers.

\begin{theorem}[Weak Morse local-to-global] \label{thmi:WMLTG}
    Let $X$ be a geodesic metric space with a bounded, quasi-geodesic combing. Then $X$ satisfies the weak Morse local-to-global property. 
\end{theorem}

A natural question to ask is the following. 
\begin{question}
 For what consequences of the MLTG property it suffices to have the weak Morse local-to-global property?
\end{question}

The difference between the weak and the (strong) Morse local-to-global is relatively subtle. Indeed, both have the form \say{If a path is $L$--locally a Morse quasi-geodesic, then it is globally a quasi-geodesic that is weakly Morse with respect to a pair of quasi-geodesic constants $(Q,q)$}. The difference is that in the weak case, the quantifiers are $\forall (Q,q) \ \exists L$, and in the strong case $\exists L \ \forall (Q,q)$. The question of whether one can upgrade the weaker property to the stronger one then becomes a question on whether one can deduce quasi-convexity for all quasi-geodesic constants by only looking at a finite number of them. To solve such a question, a well established technique is to exploit some source of compactness, and the natural object to consider is the Morse boundary. However, the Morse boundary of a metric space is compact if and only if the space is hyperbolic. Hence, we need to find a weaker notion of compactness. Our result is the following. 

\begin{theorem}\label{thmi:non-sigma}
    Let $X$ be a geodesic metric space satisfying the weak MLTG property and whose isometry group acts coboundedly on $X$. If the Morse boundary of $X$ is $\sigma$--compact then $X$ satisfies the MLTG property.
\end{theorem}

Since He, the second author and the third author showed in \cite{HeSprianoZbinden:sigmacompactnes} that a space with the MLTG property has a $\sigma$--compact Morse boundary, we obtain the following dichotomy.
\begin{corollary}
     Let $X$ be a geodesic metric space satisfying the weak MLTG property and whose isometry group acts coboundedly on $X$. Then $X$ satisfies the MLTG property if and only if the Morse boundary of $X$ is $\sigma$--compact.
\end{corollary}
In general, understanding the topology of boundaries is difficult. For instance, only recently the third author constructed the first example of a group with non-$\sigma$--compact Morse boundary using small cancellation theory \cite{Zbinden:SmallCancellation}.

\subsection{The Morse local-to-global property: a survey}\label{subsec:intro_survey} In Appendix~\ref{sec:survey} we provide a survey of recent developments of the MLTG property and upgrade some of the results about MLTG groups. We give a brief presentation here.

There are two main strengths of the MTLG property. The first is that it is invariant by quasi-isometry and satisfied by a lot of examples of interests such as Mapping Class Groups and Teichm\"uller spaces, CAT(0) spaces, fundamental groups of closed 3-manifolds, injective spaces and extra-large type Artin groups. For a more detailed list we refer to Theorem~\ref{appendix:MLTG_examples}. Notably, all the known counterexamples (Examples~\ref{counterex:1}, \ref{counterex:2}, \ref{counterex:3}) are not finitely presented. 
The second strength is that the MLTG property has a variety of strong consequences with very distinct flavours. For instance, there are combination theorems for stable subgroups of MLTG groups (Theorem~\ref{appendix:combination_stable}) which imply that if a non-cyclic, torsion-free MLTG group has a Morse ray, then it has a stable free subgroup (Corollary~\ref{appendix:MLTG_Morse_imply_free}). The latter fact is one of the key ingredients in showing that stable subgroups of MLTG groups present a growth-gap (Theorem~\ref{appendix:growth_gap}), the other being that MLTG groups satisfy excellent algorithmic and language theoretic properties (Theorem~\ref{appendix:reg_languages}). On the topological side, one can show that products of stable subgroups in certain MLTG groups are separable in the profinite topology (Theorem~\ref{appendix:product_separable}), and that the Morse boundary of MLTG groups needs to be $\sigma$--compact (Theorem~\ref{appendix:sigma_cpt}), allowing to better understand the behaviour of stationary measures on a variety of associated boundaries (Theorem~\ref{appendix:measures}).

Lastly, we want to present the results that we improve by relaxing some hypotheses. A detailed discussion can be found in Section~\ref{appendix:stability_section}. Firstly, by improving the combination theorem for stable subgroups (see Theorem~\ref{appendix:combination_stable}) we can find free subgroups in MLTG groups.

\begin{corollary}[{\cite{RussellSprianoTran:thelocal} and Lemma~\ref{appendix:improved_malnormal}, abundance of free subgroups}]\label{intro:MLTG_Morse_imply_free}
    Let $G$ be a Morse local-to-global group. If $Q$ is a non-trivial, infinite index stable subgroup of $G$, then there is an infinite order element h such that $\langle Q, h\rangle \cong Q  \ast \langle h \rangle$ and $\langle Q, h \rangle$ is stable in $G$.
\end{corollary}
The main improvement from \cite{RussellSprianoTran:thelocal} is that in the original version $G$ was required to be torsion-free, a rather restrictive assumption. In particular, Corollary~\ref{intro:MLTG_Morse_imply_free} allows to upgrade the growth-gap of stable subgroups to \emph{all} MLTG groups.
\begin{theorem}[\cite{CordesRussellSprianoZalloum:regularity} and Lemma~\ref{appendix:improved_malnormal}, growth gap]
    Let $G$ be a group with a finite generating set $S$ and with the Morse local-to-global property. Let $H < G$ be infinite and of infinite index. Then the growth rate of $H$ with respect to $S$ is strictly smaller than the growth rate of $G$ with respect to $S$.
\end{theorem}

\section*{Acknowledgments} 
We are grateful to Harry Petyt for helpful conversations on the weak Morse local-to-global property and for suggesting to investigate the connection between it and the $\sigma$-compactness of the Morse boundary. We thank Jacob Russell and Abdul Zalloum for important inputs on the appendix. Finally, we thank Carolyn Abbott and Joshua Perlmutter for comments on an early draft of the appendix. Part of this work was supported by the Research Center of Christ Church College, Oxford. 

\newpage
\section{Preliminaries}

\begin{cvn}
Unless otherwise stated, we will always assume that metric spaces are geodesic.
\end{cvn} 

Given a metric space $X$ and a subset $Y\subseteq X$, we denote the closed $r$--neighbourhood of $Y$ by $\mc{N}_r(Y)$.
\begin{definition}[Quasi-isometric embedding, quasi-isometry, quasi-geodesic]
    Given metric spaces $X,Y$, a map $f\colon X  \to Y$ is a \emph{$(\la, \ka)$--quasi-isometric embedding} if for every pair of points $x,y\in X$ we have that
   \[\frac{1}{\la}\dist_X(x,y)  - \ka \leq \dist_Y (f(x), f(y)) \leq \la \dist_X(x,y) + \ka.\]
   If, moreover, $\mc{N}_\ka(f(X)) = Y$, then $f$ is called a \emph{$(\la, \ka)$--quasi-isometry}. A \emph{$(\la, \ka)$--quasi-geodesic} is a $(\la, \ka)$--quasi-isometric embedding $\gamma \colon I \to X$, where $I\subseteq \RR$ is a closed interval.
\end{definition}

For us a \emph{path} is a continuous map $\pgot \colon [a,b] \to X$.

\begin{notation}
Let  $x\in X$ be a point, $A\subseteq X$ a subset and  $\pgot \colon [a,b] \to X$ a path. By abuse of notation, we write $x\in \pgot$ to mean that $x$ belongs to the image of $\pgot$, and $\pgot \subseteq A$ to mean that the image of $\pgot$ is contained in $A$. We denote by $\pgot^{-1}$ the path $\pgot^{-1} : [a, b]\to X$, $\pgot^{-1}(t) = \pgot(b+a - t)$. We denote by $\dlength{\pgot}$ the length of the domain of $\pgot$, $\abs{b-a}$, and  by $\length {\pgot}$ the (arc-)length of the path $\pgot $, if $\pgot$ is rectifiable.
\end{notation}

Let $u, v\in \pgot$. We say that $x\in \pgot$ \emph{lies between $u$ and $v$} if there are $s_1\leq s_2\leq s_3$ in the domain of $\pgot$ such that $\pgot(s_1) = u$, $\pgot(s_2) = x$, $\pgot(s_3)= v$. If $s_1$ is the smallest parameter so that $\pgot(s_1) = u$ and $s_3$ is the largest parameter so that $\pgot(s_3) = v$ we denote by $\pgot\vert_{uv}$ the restriction of $\pgot$ to $[s_1, s_3]$ (i.e. the maximal sub-path of $\pgot$ with image composed of points that lie between $u$ and $v$).

Note that quasi-geodesics do not need to be continuous, and can be rather complicated. However, in a geodesic metric space each quasi-geodesic can be approximated by one that is continuous and well-behaved, so that it suffices to study the latter ones. This is made precise in the following Lemma. 

\begin{lemma}[Improved quasi-geodesics \cite{BridsonHaefliger}, Lemma 1.11, \cite{Burago-Ivanov}, Proposition 8.3.4]\label{lemma:tamingqg}
Let $X$ be a geodesic metric space. For every  $(\la, \ka)$-quasi-geodesic $\gamma: [a, b]\to X$ there exists a continuous $(\la, \ka')$-quasi-geodesic $\bar{\gamma} : [a, b]\to X$ such that 
\begin{enumerate}
    \item $\gamma(a) = \bar{\gamma}(a)$ and $\gamma(b) = \bar{\gamma}(b)$;
    \item $\ka' = 2(\la + \ka)$;
    \item $\length {\bar{\gamma}|_{[t, t']}}\leq k_1 \dist (\bar{\gamma}(t), \bar{\gamma}(t'))+k_2$ for all $t, t'\in [a, b]$, where $k_1 = \la(\la+\ka)$ and $k_2 = (\la\ka' + 3)(\la+\ka)$;
    \item the Hausdorff distance between the images of $\gamma$ and $\bar{\gamma}$ is less than $\la+\ka$.
\end{enumerate}
\end{lemma}

We call such a quasi-geodesic $\bar{\gamma}$ an {\em improvement} of $\gamma$. A $(\la, \ka)$--quasi-geodesic is \emph{improved} if it is continuous and condition $(3)$ of Lemma~\ref{lemma:tamingqg} holds. 
\begin{cvn}
From now on, all quasi-geodesics are assumed to be improved.
\end{cvn}

 \subsection{Quasi-geodesics}

 In this subsection, we recall basic facts about quasi-geodesics. The results are folklore, but as we could not find a reference in the literature, we included their proofs, for the sake of completeness.

 \begin{lemma}[Containment implies finite Hausdorff distance]\label{lem:reverse_inclusion_QG_nbhd}
Let $\gamma_i$, $i=1,2$, be two $(\la, \ka)$--quasi-geodesic segments with endpoints at distance $d$. Then for all $\wmorse\geq d$ there exists $\wmorse' =(1 + 2\la^2 )\wmorse + \la\ka + \ka $ such that if $\gamma_1 \subseteq \nn_\wmorse(\gamma_2)$  then $\gamma_2\subseteq \nn_{\wmorse'}(\gamma_1)$.
\end{lemma}
\begin{proof}
\begin{figure}[ht!]
    \centering
    \includegraphics[width=.7\linewidth]{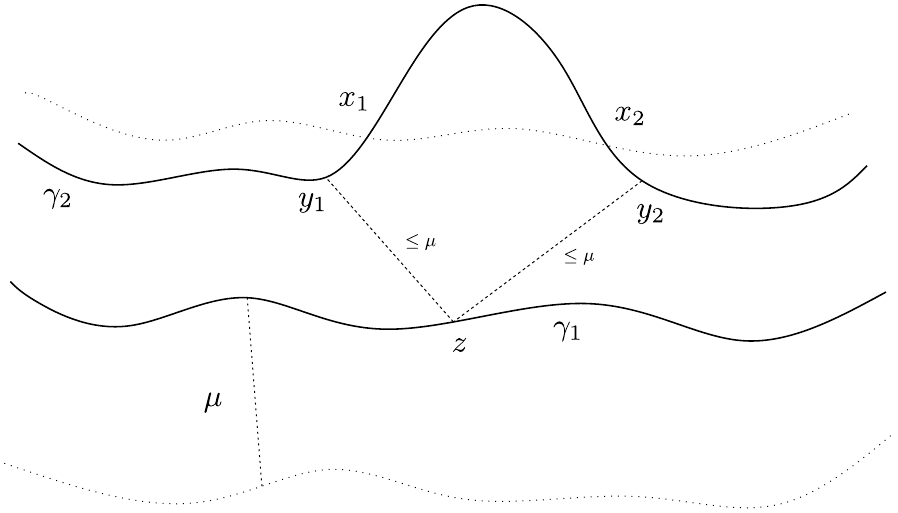}
    \caption{Proof of Lemma \ref{lem:reverse_inclusion_QG_nbhd}}
    \label{fig:proof-of-reverse-inclusion}
\end{figure}
Define an order on $\gamma_2$ by ordering its endpoints, and denote these endpoints by $a$ and $b$. 
We will bound the diameter of the connected components of $\gamma_2 \setminus \nn_\wmorse(\gamma_1)$. Consider the largest (diameterwise) such component, and let $x_1, x_2$ be its endpoints. This is depicted in Figure~\ref{fig:proof-of-reverse-inclusion}. There exists  $z\in \gamma_1$, $y_1\in \gamma_2$ before (or equal to) $x_1$, and  $y_2\in \gamma_2$ after (or equal to) $x_2$, such that $\dist (y_i, z) \leq \wmorse$. If not, $\gamma_1$ could be covered by the closed disjoint sets $\nn_\wmorse(\gamma_2\vert_{ax_1}) \sqcup \nn_\wmorse(\gamma_2\vert_{x_2 b})$, contradicting the assumption that quasi-geodesics are continuous. Thus, $\dist (y_1, y_2) \leq 2\wmorse$, and hence $\dist (m, \{y_1, y_2\}) \leq \la(2\la \wmorse +\ka) + \ka$ for all $m\in \gamma_2$ between $y_1, y_2$, yielding that $\gamma_2\subseteq \nn_{\wmorse'}(\gamma_1)$ where $\wmorse'= \la(2\la \wmorse +\ka) + \ka+\mu$. 
\end{proof}

\begin{lemma}[Appending short quasi-geodesics to a longer one]\label{lem:_wiggling_ends_of_QG}
Let $\gamma$ be a $(\la ,\ka )$--quasi-geodesic segment. Let $z_1, z_2 \in X$ and $u_1,u_2\in \gamma$ be such that
\begin{enumerate}
\item every point point $x\in \gamma$ between $u_1, u_2$ satisfies $\dist (x,z_i) \geq \dist (u_i, z_i)$.
\end{enumerate}  
Then, for all geodesics $\alpha_i$ connecting $u_i$ and $z_i$, the concatenations $\alpha_1 \ast \gamma\vert_{u_1 u_2}$ and $\gamma\vert_{u_1 u_2} \ast \alpha_2$ are $(2\la +1, \ka)$--quasi-geodesics.
If, moreover, the following condition is satisfied
\begin{enumerate}\setcounter{enumi}{1}
\item  $\dist (u_i, z_i)\leq \theta \dist (u_1, u_2)$ for some $\theta \in [0, 1/2)$;
\end{enumerate}
then the concatenation $\alpha_1 \ast \gamma\vert_{u_1 u_2}\ast \alpha_2$ is a $(\la', \ka)$--quasi-geodesic, where $\la' = \max \left(  \frac{\la +1 }{ 1- 2\theta} , 2\la +1 \right).$
\end{lemma}
\begin{proof}
Let $\eta$ denote the concatenation $\alpha_1 \ast \gamma\vert_{u_1 u_2}\ast \alpha_2$, where $\alpha_i$ are geodesics joining  $u_i$ and $z_i$, parameterized by their length, and let $\eta(s_i) = u_i$. By the triangular inequality, $\eta$ satisfies the quasi-geodesic upper inequality with constants $(\la , \ka )$. In what follows, we focus on the lower inequality. Consider two points $x = \eta(s_x)$ and $y= \eta(s_y)$, with $s_x \leq s_y$. If $x,y$ are both contained in one of the three (quasi)-geodesics composing $\eta$, then the lower inequality is satisfied.

Consider now the case  $x\in \alpha_1$ and $y\in \gamma\vert_{u_1 u_2}$ (the case $x\in \gamma\vert_{u_1 u_2}$ and  $y\in \alpha_2$ is similar).

We have that $\dist (x,y)\geq \dist (x,u_1)$, otherwise $\dist (z_1, y) < \dist (z_1, u_1)$, contradicting the second assumption. It follows that $\dist (x,y) \geq s_1-s_x = \dist (x,u_1)$, which implies that $\dist (y,u_1)\leq 2\dist (x,y)$. We can then write that
$$
s_y - s_1 \leq \la \dist (u_1, y) +\la\ka \leq  2\la \dist (x,y) +\la\ka,$$
and therefore that
$$
s_y - s_x = s_y - s_1 + s_1-s_x \leq 2\la \dist (x,y) +\la\ka + \dist (x,y) = (2\la +1) \dist (x,y) +\la\ka.     
$$
This concludes the case of a single projection.

Assume now  $x\in \alpha_1$, $y\in \alpha_2$. We have:
\begin{align*}
 \vert s_x - s_y \vert  &\leq  2\theta \dist (u_1, u_2) +\vert s_1 - s_2\vert \\ &\leq 2\theta \dist (u_1, u_2) + \la \dist (u_1, u_2) +\ka  = (2\theta  + \la ) \dist (u_1, u_2) +\ka.
\end{align*}

As $\dist (x,y)\geq (1- 2\theta ) \dist (u_1, u_2),$ we conclude that 
$$
\vert s_x - s_y \vert \leq \frac{2\theta  + \la }{ 1- 2\theta} \dist (x,y) +\la\ka \leq \frac{\la +1 }{ 1- 2\theta} \dist (x,y) +\la\ka.  
$$
\end{proof}

\begin{lemma}[Local quasi-geodesics close to quasi-geodesics are quasi-geodesics]\label{lem:local_QG_in_nbhd_is_QG}
    Let $X$ be a geodesic metric space. Then for all quasi-geodesic pairs $(\la, \ka), (\lao, \kao)$ and constants $r\geq 0$ there exists a scale $D$ and a quasi-geodesic pair $(\la', \ka')$ such that if a path $\pgot$ is contained in the $r$--neighbourhood of a $(\lao, \kao)$--quasi-geodesic $\gamma$ and is $D$--locally a $(\la, \ka)$--quasi-geodesic, then $\pgot$ is a $(\la', \ka')$--quasi-geodesic. 
\end{lemma}

\begin{proof}
Since the statement for paths $\pgot$ with unbounded domain follows directly from the statement for paths $\pgot$ with bounded domain, without loss of generality we can assume that the domain of $\pgot$ is bounded and equal to $[0, T]$ for some $T\geq 0$. By increasing $r$ a uniform amount (and possibly passing to a subsegment of $\gamma$) we can assume that the closest point projections of the endpoints of $\pgot$ on $\gamma$ are exactly the endpoints of $\gamma$.

Since $\pgot$ is a $D$--local quasi-geodesic, the map $\pgot$ satisfies the upper coarse Lipschitz inequality. We focus on the lower one. For each $i\geq 0$ where defined, let $x_i = \pgot\left(i \frac{D}{2}\right)$ and let $y_i = \gamma(s_i)$ be a closest point projection of $x_i$ onto $\gamma$. Since $\pgot$ is a quasi-geodesic at scale $D$, $\abs{s_{i+1}- s_{i}}\geq D/c - c$ for some constant $c$ depending on $r, (\la, \ka)$ and $ (\lao, \kao)$ but not on $D$. 

By the choice of parameterization, we can assume that $s_1 - s_0\geq 0$. We will show inductively that for large enough $D$, $s_{i+1} - s_i\geq 0$ for all $i$ and hence $s_{i+1}- s_i\geq D/c - c$, which concludes the proof since it yields a linear lower bound on the lower coarse Lipschitz inequality for $\pgot$.

Assume that $s_{i} - s_{i-1}\geq 0$. If $s_{i-1}\leq s_{i+1}\leq s_{i}$, then by Lemma~\ref{lem:reverse_inclusion_QG_nbhd}, there exists $(i-i)\frac{D}{2}\leq t \leq i\frac{D}{2}$ such that $\dist (\pgot(t), \gamma(s_{i+1}))\leq r'$, where $r'$ depends on $r$ but not on $D$. Consequently, $\dist (\pgot(t), \pgot((i+1)\frac{D}{2}))$ is bounded by a constant independent of $D$. For large enough $D$, this contradicts $\pgot$ being a $(\lao, \kao)$--quasi-geodesic at scale~$D$. 

Similarly, if $s_{i+1}\leq s_{i-1}\leq s_i$ we get a contradiction by swapping $(i-1)$ and $(i+1)$. Therefore, we have to have $s_{i+1}> s_i$, which concludes the induction and hence the proof.
\end{proof}


\subsection{Combings}\label{sec:combings}
Consider a metric space $X$ and two constants $\lao \geq 1$ and $\kao \geq 0$. 
A \emph{$(\lao,\kao)$--quasi-geodesic combing} is a way of assigning to  every ordered pair of points $(x, y)\in X\times X$, a $(\lao,\kao)$--quasi-geodesic $\col_{xy}$ connecting them. The quasi-geodesics $\col_{xy}$ are called \emph{combing lines}. When needed, we assume the  quasi-geodesics to be extended to $\R$ by constant maps. 

We say that a $(\lao, \kao)$--quasi-geodesic combing is \emph{bounded} if 
\begin{equation*}\label{eq:dist1}
 \dist (\col_{xy_1}(t), \col_{xy_2}(t))\leq \kao\dist (y_1,y_2)+\kao,    
\end{equation*} for all $t\in \R$ and $x, y_1,y_2$ in $X$.

\subsection{Morse and weak Morse property}
The main focus of our paper is the exploration of the weak Morse property from the perspective of local-to-global phenomena. We start from the definition of weak Morse property. This should be thought of as a form of convexity with respect to only a given type of quasi-geodesics.
\begin{definition}[Weak Morse]\label{defn:weak_Morse}
    Let $(Q, q)$ be a quasi-geodesic pair and let $\wmorse \geq 0$. A quasi-geodesic $\gamma$ is {\emph{}}$(Q, q, \wmorse)$--weakly Morse if any $(Q,q)$--quasi-geodesic $\eta$ with endpoints $\gamma(s)$ and $\gamma(t)$ satisfies 
    \[
    \eta \subseteq \mc{N}_{\wmorse}(\gamma\vert_{[s, t]}).
    \]
\end{definition}

Note that a continuous quasi-geodesic is $(1,0,0)$--weakly Morse if it is convex. Often, however, it is useful to consider quasi-convexity with respect to \emph{all} quasi-geodesic pairs. We say that a \emph{Morse gauge} is a function $M \colon \mathbb{R}_{\geq1} \times \mathbb{R}_{\geq0} \to \mathbb{R}_{\geq 0}$, and we define Morse quasi-geodesics as follows. 

\begin{definition}[Morse quasi-geodesic]\label{defn:Morse_QG}
A quasi-geodesic $\gamma$ is \emph{$M$--Morse}, for a Morse gauge $M$, if any $(Q, q)$--quasi-geodesic $\eta$ with endpoints $\gamma(s)$ and $\gamma(t)$ satisfies 
    \[
    \eta \subseteq \mc{N}_{M(Q,q)}(\gamma\vert_{[s,t]}).
    \]
\end{definition}

In \cite{RussellSprianoTran:thelocal}, Russell, Tran and the second author started the study of spaces with the property that paths that are locally Morse quasi-geodesics are globally Morse quasi-geodesics. For an overview of the subject and examples, we refer the reader to the Appendix (Section~\ref{sec:survey}). 

\begin{definition}
    A path $\pgot \colon [a, b] \to X$ \emph{satisfies a property $(P)$ at scale $L$} if for every $t_1, t_2 \in [a,b]$ with $\vert t_1- t_2 \vert \leq L$ the restriction $\pgot\vert_{[t_1, t_2]}$ has the property $(P)$. The quantity $L$ is called the \emph{scale}.
\end{definition}
Sometimes we use the more suggestive expressions \emph{$\pgot$ $L$--locally satisfies $(P)$}, or \emph{$\pgot$ is $L$--locally $(P)$} to mean that $\pgot$ satisfies a property $(P)$ at scale $L$. A \say{local-to-global property} is a condition requiring that if a path is $L$--locally $(P)$ for $L$ large enough, then it is globally $(P')$. An inspiring example for the topic is a theorem of Gromov stating that a geodesic metric space is hyperbolic if and only if all of its local quasi-geodesics are global quasi-geodesics \cite{Gromov(1983)}.

\begin{definition}[Weak MLTG property]\label{defn:wmltg}
 A metric space $X$ satisfies the \emph{weak Morse local-to-global} (\emph{WMLTG} for short) \emph{property} if the following holds. For all quasi-geodesic pairs $(\la, \ka)$ and $ (Q,q)$, and every Morse gauge $M$ there exists a scale $L$, a quasi-geodesic pair $(\la', \ka')$ and a constant $\wmorse \geq 0$ such that every path that is $L$--locally an $M$--Morse $(\la, \ka)$--quasi-geodesic is globally a $(\la', \ka')$--quasi-geodesic that is $(Q,q, \wmorse)$--weakly Morse.    
\end{definition}

The definition of the Morse local-to-global property is very similar, but with the stronger conclusion of being Morse instead of weakly Morse. 
\begin{definition}[MLTG property]\label{defn:MLTG}
 A metric space $X$ satisfies the \emph{Morse local-to-global} (\emph{MLTG} for short) \emph{property} if the following holds. For every  quasi-geodesic pair $(\la, \ka)$ and Morse gauge $M$ there exists a scale $L$, a quasi-geodesic pair $(\la', \ka')$ and a Morse gauge $M'$ such that every path that is $L$--locally an $M$--Morse $(\la, \ka)$--quasi-geodesic is globally an $M'$--Morse $(\la', \ka')$--quasi-geodesic.
\end{definition}

In Section~\ref{sec:WMLTG}, we will show that under suitable compactness assumptions of the Morse boundary, the WMLTG property can be upgraded to the MLTG property. 

\subsection{Morse property via middle recurrence}

In practical use, there is a much more effective formulation of the Morse property in terms of paths of controlled length, instead of quasi-geodesics. The advantage is that controlling the length of a concatenation of paths in terms of the original length is easy, whereas determining whether a concatenation of quasi-geodesics is a quasi-geodesic is hard. 

In this subsection, we use the notion of middle recurrence to establish a lemma that explicitly relates the Morse property and the lengths of certain paths. The reader interested only in the broad lines of the proof can safely take Lemma~\ref{lemma:mid-recurrence-consequence'} for granted and skip the middle recurrence altogether, as it will not be used in the rest of the paper. Middle recurrence was introduced in \cite[Proposition 3.24]{DrutuMozesSapir} and \cite[Proposition 1]{DrutuMozesSapir-corr} and further studied in \cite{aougabdurhamtaylor:pulling}.

\begin{definition}[$\mathbf{t}$--middle]
Let $\gamma$ be a quasi-geodesic and $a,b\in \gamma$. For $\mathbf{t}\in \left( 0,\frac{1}{2} \right)$, the \emph{$\mathbf{t}$--middle} of $\gamma\vert_{ab}$ is the set of $x\in\gamma$ lying between $a,b$ such that  $\min \{ \dist (x,a), \dist (x,b)\} \geq \mathbf{t}\cdot \dist (a,b)$. We denote the $\mathbf{t}$--middle of $\gamma\vert_{ab}$ as $\gamma\vert_{\mathbf{t}\cdot ab}$. When $a,b$ are the endpoints of $\gamma$, we denote the $\mathbf{t}$--middle simply by $\gamma\vert_{\mathbf{t}}$.
\end{definition}

\begin{definition}
Let $\gamma$ be a quasi-geodesic and $\bt \in \left(0, \frac{1}{2} \right)$.
We say that $\gamma$ is \emph{$\bt$--middle recurrent} if there is a
function $\mf{m}_{\bt}\colon \mathbb{R}_{+} \to \mathbb{R}_{+}$  so that any path $\mf{p}$ with endpoints $a,b \in \gamma$ and $\length{\pgot}\leq c\cdot \dist (a,b)$ satisfies \[\pgot \cap \nn_{\mf{m}_{\bt}(c)}(\gamma\vert_{\mathbf{\bt}\cdot ab})\neq \emptyset. \] The function $\mf{m}_{\bt}$ is called the \emph{$\bt$--recurrence function} of $\gamma$.

We say that a quasi-geodesic $\gamma$ is \emph{middle recurrent} if it is \emph{$\bt$--middle recurrent} for some fixed $\bt \in \left(0, \frac{1}{2} \right)$. 
\end{definition}

Note that subpaths of a $\bt$--middle recurrent path are $\bt$--middle recurrent with respect to the same recurrence function.

The relevance of middle recurrence is established in the following theorem.
\begin{theorem} [\cite{DrutuMozesSapir,DrutuMozesSapir-corr,aougabdurhamtaylor:pulling}]\label{thm:middleandmorse}
Let $\gamma$ be a quasi-geodesic in a geodesic metric space $X$. Then $\gamma$ is Morse if and only if it is middle recurrent. Moreover, its recurrence function can be bounded from above only in terms of its Morse gauge, and vice versa.
\end{theorem}

For our goals, it will suffice to restate Theorem~\ref{thm:middleandmorse} in the following form that, as mentioned, does not in fact involve the definition of middle recurrence. 

\begin{lemma}\label{lemma:mid-recurrence-consequence'}
    Let $M$ be a Morse gauge and let $(\la, \ka)$ be a  quasi-geodesic pair. Let $\K, \sigma, \delta$ be linear functions. Then there exists a constant $D_0\geq 0$ such that for all $D\geq D_0$ there exists a constant $\ell_0\geq 0$ such that for all $\ell\geq \ell_0$ the following holds. If $\gamma$ is a $(\la, \ka)$--quasi-geodesic which is $M$--Morse, then there does not exist a path $\pgot$ with endpoints $\gamma(t_1)$ and $\gamma(t_2)$ satisfying the following:
    \begin{enumerate}
        \item $\ell\leq t_2 - t_1 \leq \sigma(\ell)$,\label{C1'}
        \item $\length{\pgot}\leq \K(\ell)$,\label{C2'}
        \item $\dist (\pgot, \gamma \vert_{[t_1 + \delta(D), t_1 + \ell - \delta(D)]})\geq D $.\label{C3'}
    \end{enumerate}
\end{lemma}
\begin{proof}
Assume by contradiction that there exists a path $\pgot$ satisfying \eqref{C1'}, \eqref{C2'} and \eqref{C3'}. Let $a = \gamma(t_1)$ and $b = \gamma(t_1 + \ell)$. Modify $\pgot$ by appending the segment $\gamma^{-1}\vert_{[b, \gamma(t_2)]}$. Note that the new path $\pgot$ satisfies Condition~\eqref{C3'} for a linear function $\delta' = \delta' (\delta, \la, \ka)$. By Lemma~\ref{lemma:tamingqg}, we have that the new path $\pgot$ satisfies $\length{\pgot}\leq \K'(\ell)$ for a linear function $\K'= \K'(\K, \sigma, \la, \ka)$. Thus there exist $\ell_1,c$ only depending on $\K', \la, \ka$ so that if $t_2 - t_1\geq \ell_1$, we have $\length{\pgot}\leq c \cdot \dist (a, b)$. According to 
Theorem~\ref{thm:middleandmorse}, there exist $\tgot\in (0, \frac{1}{2})$ and a middle recurrence $\mf{m}_{\tgot}$, both depending only on $M$, such that $\gamma$ is $\mf{m}_{\tgot}$--middle recurrent. Thus, $\pgot$ intersects the $\mf{m}_{\tgot}(c)$--neighbourhood of the $\tgot$--middle $\gamma_{\tgot\cdot ab}$. Choose $D_0$ as $\mf{m}_\tgot(c)  + 1$ and let $D\geq D_0$. We'll show that for $\ell$ large enough we have $\gamma \vert_{\tgot \cdot ab} \subseteq \gamma \vert_{[t_1 + \delta(D), t_1 + \ell - \delta(D)]}$, contradicting \eqref{C3'}. Indeed, and $\gamma \vert_{\tgot \cdot [a,b]}$ has a lower bound that grows linearly in $\ell$, whereas the Hausdorff distance between $a$ and $\gamma \vert_{[t_1, t_1 + \delta(D)]}$ has a uniform upper bound. A similar observation for $b$ and $\gamma \vert_{[t_1 + \ell -\delta(D), t_1 + \ell]}$ shows that for $D\geq D_0$ there exists $\ell_2 = \ell_2(D)$ so that for all $\ell\geq \ell_2$ it holds  $\gamma \vert_{\tgot \cdot ab} \subseteq \gamma \vert_{[t_1 + \delta(D), t_1 + \ell - \delta(D)]}$. Choosing $\ell_0 = \max\{\ell_1, \ell_2\}$ concludes the proof.
\end{proof}

\section{Globalization of quasi-geodesics}

\begin{lemma}\label{lem:Morse_for_combing_lines}
    Let $X$ be a geodesic metric space with a bounded $(\lao, \kao)$--quasi-geodesic combing. For every $\la\geq 1$ and $\ka \geq 0$, and for every $\wmorse\geq 0$ there exists $D\geq 0$ and $\mu'\geq 0 $ such that every path $\pgot$ that is $D$--locally a $(2\lao +2, \kao, \wmorse)$--weakly Morse $(\la,\ka)$--quasi-geodesic is contained in the $\mu'$--neighbourhood of the combing line between its endpoints. 
\end{lemma}
\begin{proof}
    We can assume that $\pgot$ is parameterized as $\pgot \colon[0,a] \to X$ for some $a$. Let $k = \lceil a \rceil$ and let $\qgot_i$ be the combing line between $\pgot(0)$ and $\pgot(i)$ for $i <k$, and $\qgot_k$ be the combing line between $\pgot(0)$ and $\pgot (a)$. 

    Write $\mu' = \theta D$ and assume that  $\pgot$ is not in the $\theta D$--neighbourhood of $\qgot_k$. By a (coarse) continuity argument there exists $\qgot_i$ such that the corresponding subpath $\pgot\vert_{[0,i]}$ is  in the $\theta D + O(1)$--neighbourhood of $\qgot_i$ but not in the $\theta D$--neighbourhood of $\qgot_i$.     

We fix such a $\qgot_i$ and from now on we abuse notation and simply denote $\qgot_i$ and  $\pgot\vert_{[0,i]}$ as $\qgot$ and  $\pgot$, respectively. Let $\pgot(s)$ be a point at distance at least $\theta D$ from $\qgot$. Let $\rho \in (0,1/2)$ be another small constant to be determined. Let $t_1 = \max (0,  s -\rho D)$ and $t_2 = \min (a, s + \rho D)$. We can assume that $a\geq \rho D$ and hence $D\geq 2\rho D\geq t_2 - t_1\geq \rho D$.

For $i = 1, 2$, let $u_i\in \qgot $ be a closest point to $\pgot(t_i)$.  Since $\dist (u_i, \pgot(t_i)) \leq \theta  D + O(1)$, we have that 
\begin{align*}
   \dist(u_1, u_2)&\geq \dist (\pgot(t_1), \pgot(t_2)) - 2\theta D - 2\mathcal O(1)\geq \frac{(t_2 - t_1)}{\la} - \ka -  2\theta D - 2\mathcal O (1)\\
   &\geq \left ( \frac{\rho}{\lambda} - 2\theta\right) D - \ka - 2\mathcal O (1).
\end{align*}

Here we used that $\pgot$ is $D$--locally a $(\la, \ka)$--quasi-geodesic. Whence $\dist (u_i, z_i) \leq \frac{1}{6} \dist (u_1, u_2)$ if $\theta < \frac{\rho}{8\la}$ and $D$ is large enough. By Lemma \ref{lem:_wiggling_ends_of_QG}, it follows that if $\alpha_i$ are geodesics connecting $\pgot(t_i)$ to $u_i$, then the concatenation $\gamma = \alpha_1 \ast \qgot\vert_{u_1 u_2}\ast \alpha_2$ is a $(\lao', \kao)$-quasi-geodesic, where $\lao'= 2\lao +2$, joining two points on $\pgot$ at parameterized distance less than $D$, implying that $\gamma$ is in the $\wmorse$--neighbourhood of $\pgot \vert_{[t_1, t_2]}$. By Lemma~\ref{lem:reverse_inclusion_QG_nbhd}, we consequently have that $\pgot \vert_{[t_1,t_2]}$ and in particular $\pgot (s)$ is in the $r$--neighbourhood of $\gamma$ for some $r$ only depending on $\wmorse, \lao, \kao, \la$ and $\ka$. For $\mu' = \theta D$ larger than $r$, $\pgot(s)$ is not in the $r$ neighbourhood of $\alpha_1$ and $\alpha_2$ and hence not in the $r$--neighbourhood of $\qgot$, a contradiction.
\end{proof}

Combining Lemmas~ \ref{lem:reverse_inclusion_QG_nbhd}, \ref{lem:local_QG_in_nbhd_is_QG}, and \ref{lem:Morse_for_combing_lines} we obtain the following. 

\begin{theorem}\label{thm:local-to-global-qg}
Let $X$ be a geodesic metric space with a bounded $(\lao, \kao)$--quasi-geodesic combing.
For every quasi-geodesic pair $(\la, \ka)$ and $\wmorse \geq 0$, there exists a scale $D\geq 0$, and quasi-geodesic constants $(\la', \ka')$ such that every path $\pgot$ that $D$--locally is a $(2\lao +2, \kao, \mu)$--weakly Morse $(\la,\ka)$--quasi-geodesic, is a global $(\la',\ka')$--quasi-geodesic.

Moreover, for every pair of points $u$ and $v$ in the image of $\pgot$, the subpath $\pgot|_{uv}$ and the combing line $\col_{uv}$ are within Hausdorff distance at most $\mu''  = \mu'' (\la, \ka, \lao, \kao, \mu)$ of each other. 
\end{theorem}

\section{Combings and the weak Morse-local-to-global property}
The main goal of this section is to prove Theorem~\ref{thm:mltg}, namely showing that a metric space equipped with a bounded combing satisfies the weak Morse local-to-global property. 

We begin by a few technical lemmas.

\begin{lemma}\label{lemma:quasi-geodesics_stay}
Let $(\la, \ka)$ be a quasi-geodesic pair. There exists a linear function $\delta$ such that the following holds. Let $\gamma_1 : [0, T_1]\to X, \gamma_2 : [ 0 , T_2]\to X$ be $(\la, \ka)$-quasi-geodesics such that $\gamma_2\subseteq \nn_D (\gamma_1)$.

 Assume that $0\leq t_1\leq s_1\leq T_1$ and $0\leq t_2\leq s_2\leq T_2 $ are such that the points $x_i = \gamma_i(t_i)$ and $ y_i = \gamma_i(s_i), i=1,2,$ satisfy $\dist (x_1, x_2)\leq D$ and $\dist (y_1, y_2)\leq D$. Then $\dist (\gamma_2\vert_{[t_2, s_2]}, \gamma_1(t)) > D$ for all $t\in[0, T_1]\setminus [t_1 -\delta(D), s_1+\delta(D)]$. 
\end{lemma}
\begin{proof} 
We start by proving that $\dist (\gamma_2\vert_{[t_2, s_2]}, \gamma_1\vert_{[0, t_1 - \delta(D)]}) > D$ for $t\in [0, t_1 - \delta(D)]$, where $\delta(D) = \la(\epsilon(D) + 2D) + \ka$ and $\epsilon$ is a linear function as defined below. 

\begin{figure}
    \centering
    \includegraphics[width=0.7\linewidth]{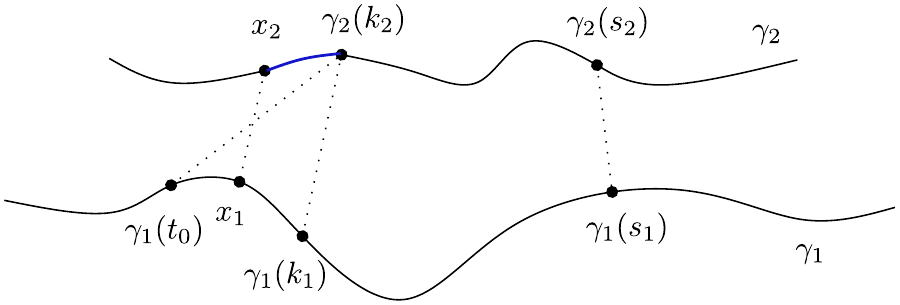}
    \caption{Proof of Lemma~\ref{lemma:quasi-geodesics_stay}. Dotted lines have length at most $D$ and every point on the blue segment has distance at most $\epsilon(D)$ from $x_2$.}
    \label{fig:lemma4.1}
\end{figure}

Let $t_2\leq k_2\leq s_2$ be the largest index such that there exists $t_0\leq t_1$ with $\dist (\gamma_2(k_2), \gamma_1(t_0)) \leq D$. By maximality of $k_2$ and continuity of quasi-geodesics, there exists $k_1 \geq t_1$ such that $\dist (\gamma_2(k_2), \gamma_1(k_1))\leq D$. This is depicted in Figure~\ref{fig:lemma4.1}. In particular, $\dist (\gamma_1(t_0), \gamma_1(k_1))\leq 2D$, and hence $k_1 - t_0 \leq 2\la D +  \ka$. Using $t_0 \leq t_1 \leq k_1$, the fact that $\gamma_1$ and $\gamma_2$ are $(\la, \ka)$--quasi-geodesics and the triangle inequality, we get that $k_2 - t_2$ are linearly bounded from above and hence there exists a linear function $\eps$ such that 
\[
\dist (\gamma_2(s), x_2)\leq \eps(D),
\]
for all $t_2\leq s \leq k_2$. If $t < t_1 - \la(\epsilon(D)+ 2D ) + \ka$. Then using that $\gamma_1$ is a $(\la, \ka)$--quasi-geodesic, we get that $\dist (\gamma_1(t), x_1) \geq \epsilon(D)+ 2D$, so by the triangle inequality, $\dist (\gamma_1(t), x_2) \geq \epsilon(D) + D$. We conclude that $\dist (\gamma_1\vert_{[0, t_1 - \delta(D)]}, \gamma_2\vert_{[t_2, s_2]}) >D$. The proof that $d(\gamma_2\vert_{[t_2, s_2]},\gamma_1\vert_{[s_1+\delta(D), T_1]})> D$ is analogous. 
\end{proof}

Given two quasi-geodesics starting at the same point, but ending far away from each other, we want to define the point at which they ``exit their $D$--neighbourhoods for sufficently'' long. This is made precise in the following definition of an exit point. The following results then describe the properties of those exit points.

\begin{definition}\label{def:exit-point}
Let $\gamma : [0, T] \to X$ and $\eta :[0, T']\to X$ be $(\la, \ka)$--quasi-geodesics starting at a point $x = \gamma(0) = \eta(0)$. We say that a point $y = \gamma(t)$ on $\eta$ is a $(D, \ell)$--exit point of $(\eta, \gamma)$ if there exists a constant $\exit\in [0, T]$ with $\dist (y, \gamma(\exit))\leq D$ and such that 
\begin{align*}
    \dist (\eta\vert_{[t, T']}, \gamma \vert_{[0, \exit+\ell]}) \geq D.
\end{align*}
A minimal such $\exit$ is called the \emph{exit moment} of $\gamma(t)$. This is depicted in Figure~\ref{fig:exit-point}. We call a $(D, \ell)$--exit point whose exit moment is minimal amongst all exit moments a \emph{minimal} $(D, \ell)$--exit point.
\end{definition}

\begin{figure}[ht!]
    \centering
    \includegraphics[width=.7\linewidth]{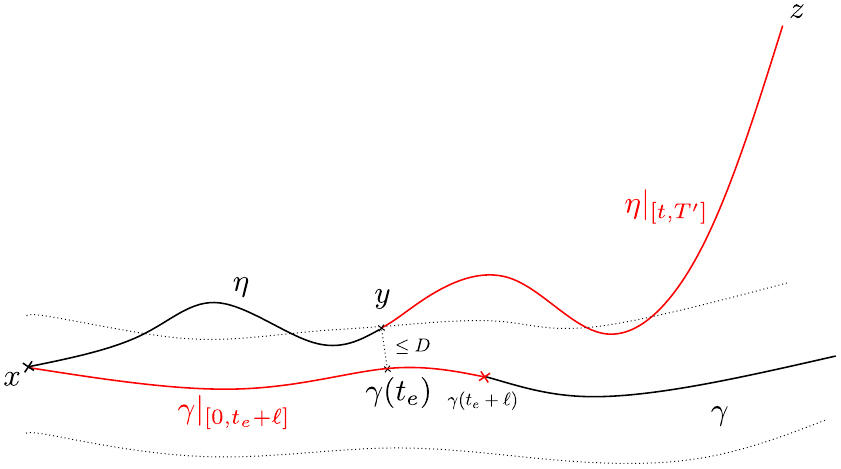}
    \caption{The point $y = \eta(t)$ is a $(D, \ell)$--exit point.}
    \label{fig:exit-point}
\end{figure}

Observe that if $\dist (\eta(T'), \gamma) > D$, then, by continuity, a $(D, \ell)$--exit point exists, for example the last point on $\eta$ in the $D$--neighbourhood of $\gamma$.

The following lemma states that if $\exit$ is the exit moment of a minimal $(D, \ell)$--exit point, then $\gamma\vert_{[0, \exit]}$ stays close to $\eta$.

\begin{lemma}   \label{lemma:minimal_exit_points}
    Using the notation of Definition \ref{def:exit-point}. If $\eta(t)$ is a minimal $(D, \ell)$--exit point of $(\eta, \gamma)$ with exit moment $\exit$, then for every $s$ with $0\leq s\leq \exit$ there exists $s'$ with $s  \leq s' \leq s+\ell$ such that $\dist (\gamma(s'), \eta)\leq D$.
\end{lemma}

\begin{proof}
  Let $t'\in [0, T']$ be minimal such that $\dist (\eta\vert_{[t', T']}, \gamma\vert_{[0, s+\ell]})\geq D$. Such a point exists since $\eta(t)$ is a $(D, \ell)$--exit point with exit moment $\exit \geq s$. By continuity of $\eta$, $\dist (\eta(t'), \gamma\vert_{[0, s+\ell]}) = D$. Let $0\leq u\leq s+\ell$ be minimal such that $\dist (\eta(t'), \gamma(u)) = D$.  If $u < s$, then $\eta(t')$ is a $(D, \ell)$--exit point with exit moment $u < \exit$, contradicting the minimality of $\exit$. Hence $s\leq u \leq s+\ell$, implying that the lemma holds, since for example $s' = u$ works.
\end{proof}

\begin{corollary}\label{coro:Hauss_distance_exit_point}
 In the notation of Definition \ref{def:exit-point}, let $(\la, \ka)$ be the quasi-geodesic constants of the quasi-geodesics $(\eta, \gamma)$. There exists a linear function $\nu$ $= \nu(\la, \ka)$ such that if $\eta(t)$ is a minimal $(D, \ell)$--exit point of $(\eta, \gamma)$ with exit moment $\exit$, then 
 \[
 d_{\mathrm{Haus}}(\gamma \vert_{[0,\exit]}, \eta \vert_{[0, t]}) \leq \nu (\ell + D).
 \]
\end{corollary}
\begin{proof}
    By Lemma \ref{lemma:minimal_exit_points}, for every $s\in [0,\exit]$ we have that $\dist (\gamma(s), \eta) \leq D + \la \ell + \ka$. Lemma~\ref{lem:reverse_inclusion_QG_nbhd} concludes the proof.
\end{proof}

\begin{lemma}\label{lemma:nice_behavior_before_exit} 
    In the notation of Definition \ref{def:exit-point}, let $(\la, \ka)$ be the quasi-geodesic constants of the quasi-geodesics $(\eta, \gamma)$. There exists a linear function $\epsilon = \epsilon(\la, \ka)$ such that if $y =\eta(t)$ is a minimal $(D, \ell)$--exit point of $(\eta, \gamma)$ with exit moment $\exit$, and $0\leq s\leq \exit$ is such that $\dist (\gamma(s), \eta(t'))\leq D$ for some $t'\in [0, t]$, then
    \begin{align*}
        \dist (\gamma \vert_{[0, s - \epsilon(D+\ell)]}, \eta \vert_{[t', T']})\geq D.
    \end{align*}
\end{lemma}
\begin{proof}
By Corollary~\ref{coro:Hauss_distance_exit_point} we have $d_{\mathrm{Haus}}(\gamma \vert_{[0,\exit]}, \eta \vert_{xy}) \leq \nu(D + \ell) $. Let $D' = \nu(D+\ell) + D+\ell$. By Lemma~\ref{lemma:quasi-geodesics_stay}, with the pair $\{\gamma(s), \eta(t')\}$ having the role of $\{x_i\}$ in the Lemma and $ \{\gamma(\exit), \eta(t)\}$ the role of $\{y_i\}$, there exists a linear function $\delta$ such that $\dist (\eta \vert_{[t', t]}, \gamma\vert_{[0, s-\delta(D')]}) \geq D'$. Since $D'\geq D$, and since the definition of exit point allows us to estimate $\dist (\eta \vert_{[t, T']}, \gamma\vert_{[0, s - \delta(D')]})$, the result follows for the function $\epsilon$ defined via $\epsilon(x) = \delta(\nu(x)+x)$.
\end{proof}

We are now ready to prove Theorem~\ref{thm:mltg}.

\begin{theorem}\label{thm:mltg}  Let $X$ be a geodesic metric space equipped with a bounded $(\lao, \kao)$--quasi-geodesic combing. Let $(\la,\ka)$ and $(Q, q)$ be quasi-geodesic pairs, and let $M$ be a Morse gauge. There exist constants $L, N$ such that the following holds. 
    Let $\gamma$ be a $(\la,\ka)$--quasi-geodesic that at scale $L$ is $M$--Morse. Then $\gamma$ is $(Q, q, N)$--weakly Morse. 
\end{theorem}
\begin{proof}
    In this proof, we can and will assume that $\la\geq \lao$ and $\ka \geq \kao$.
    
 \textbf{Outline of the proof:} We will show that for large enough constants $N_1, N_2$ the following holds. If $\gamma$ is a $(\la, \ka)$--quasi-geodesic that $L$-locally is $M$--Morse, then every $(Q,q)$-quasi-geodesic with endpoints in the closed $N_1$-neighbourhood of $\gamma$ stays in the $N_2$-neighbourhood of $\gamma$. We do so by contradiction, assuming that the property does not hold and then finding a path $\pgot$ with endpoints on $\gamma$ that contradicts the Morse property of $\gamma$. We describe the role of the most important constants involved in the proof, the precise definitions of which are in the Setup paragraph.
 \begin{itemize}
     \item $D$  describes the size of the neighbourhood the path $\pgot$ does not intersect. It should be thought of as small.
     \item $\ell$ morally is the scale where we have the Morse property, and should be thought of as of medium size.
     \item $L$ is the actual scale for the Morse property, after coarseness is taken into account. More precisely,  $L=\sigma (\ell)$, for the linear function $\sigma$ (see Claim~\ref{claim:2}). 
     \item $N_1$ determines a neighbourhood of $\gamma$. It should be thought of as large.
     \item $N_2$ determines a very large neighbourhood that will be neighbourhood of the weak Morse property. 
 \end{itemize}

\begin{figure}
\[\begin{tikzcd}
	{(\kao, \lao)} & {(\ka, \la)} && {(Q, q)} && M \\
	\\
	&& {\delta, \epsilon, \sigma} & \chi \\
	&&& {(D, \ell)} \\
	&& {N_1} \\
	&& {N_2}
	\arrow[tail reversed, no head, from=1-2, to=1-1]
	\arrow[tail reversed, no head, from=3-3, to=1-2]
	\arrow[from=3-3, to=4-4]
	\arrow[tail reversed, no head, from=3-4, to=1-2]
	\arrow[tail reversed, no head, from=3-4, to=1-4]
	\arrow[tail reversed, no head, from=4-4, to=1-6]
	\arrow[tail reversed, no head, from=4-4, to=3-4]
	\arrow[tail reversed, no head, from=5-3, to=1-2]
	\arrow[tail reversed, no head, from=5-3, to=1-4]
	\arrow[tail reversed, no head, from=5-3, to=4-4]
	\arrow[tail reversed, no head, from=6-3, to=5-3]
\end{tikzcd}\]
\caption{Dependency of constants in the proof of Theorem~\ref{thm:mltg}. An arrow from $x$ to $y$ implies that $y$ depends on $x$.}
\label{fig:dependency}
\end{figure}
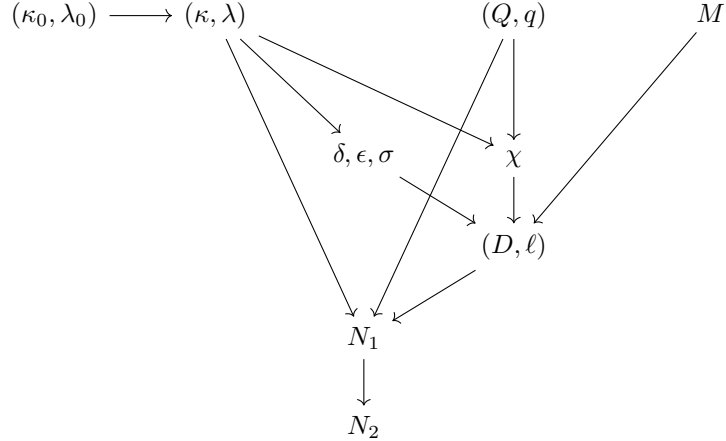

    \textbf{Setup:} We now provide a more detailed setup of the constants and functions involved in the proof. To reassure the reader, we depict the dependencies between the different constants in Figure~\ref{fig:dependency}. Let $\delta$, respectively $\epsilon$, be the linear functions obtained by applying Lemma \ref{lemma:quasi-geodesics_stay}, respectively Lemma~\ref{lemma:nice_behavior_before_exit}, to the quasi-geodesic constants $(\la, \ka)$.  Define $\sigma$ as $\sigma(x) = \epsilon(2x) +x$, and let $\K$ be the linear function defined in Claim~\ref{claim:6}.  Lemma~\ref{lemma:mid-recurrence-consequence'} applied to $(\la, \ka, \K, \sigma, \delta)$ yields $D_0$ so that choosing $D = \max\{D_0, M(3\lao +3, \kao)\}$ yields a constant $\ell_0$. Let $\ell = \max\{\ell_0, \ell_1\}$, where $\ell_1 = \ell_1(\lao, \kao, \la, \ka, M)$ is defined in Claim~\ref{claim:0}. The neighbourhoods $N_1, N_2$ will be very large compared to all the above constants, and we postpone their precise definitions. Finally, as mentioned set $L = \sigma(\ell)$.

    Let $\gamma$ be a $(\la, \ka)$--quasi-geodesic that is $L$--locally $M$--Morse. Assume that there exist $(Q, q)$--quasi-geodesics with endpoints in the $N_1$--neighbourhood of $\gamma$ which do not stay in the $N_2$--neighbourhood of $\gamma$. Let $\eta$ be the shortest (in terms of length of its domain) such $(Q,q)$-quasi-geodesic. Since $\eta$ is the shortest such quasi-geodesic, we have $\dist (\eta, \gamma)= N_1$. Let $z, z'$ be the endpoints of $\eta$ and let $x$ and $x'$ be points on $\gamma$ closest to $z$ and $z'$ respectively. We may assume that $\gamma: [0, T]\to X$ and $\eta: [0, T']\to X$ with $\gamma(0) = x, \gamma(T) = x', \eta(0) = z$ and $\eta(T') = z'$.

    Let $\net= \lfloor{T/(2\ell)}\rfloor$. 
    Since $\ell$ does not depend on $N_i$, given $N_1, m$, there exists $N_2$ so that $\net \geq m$. For each $0\leq i \leq \net$, define 
    \begin{align*}
        z_i &= \eta(i T'/\net).
    \end{align*}
    Fix a geodesic $[z', x']$ between $z'$ and $x'$. Define $\nto =\net+ \lfloor\length {[z', x']} - D \rfloor$. For each $\net+1\leq i\leq \nto$ define $z_i$ as the point on $[z', x']$ with distance $i - \net$ from $z'$.

    For all $0\leq i \leq \nto$, let $\col_i : [0, T_i]$ be the combing line from $x$ to $z_i$ and let $y_i = \col_i(s_i)$ be a minimal $(D, \ell)$--exit point for the pair $(\gamma, \col_i)$ with exit moment $t_i$, and let $x_i = \gamma(t_i)$.

\begin{figure}[ht!]
    \centering
    \includegraphics[width=\linewidth]{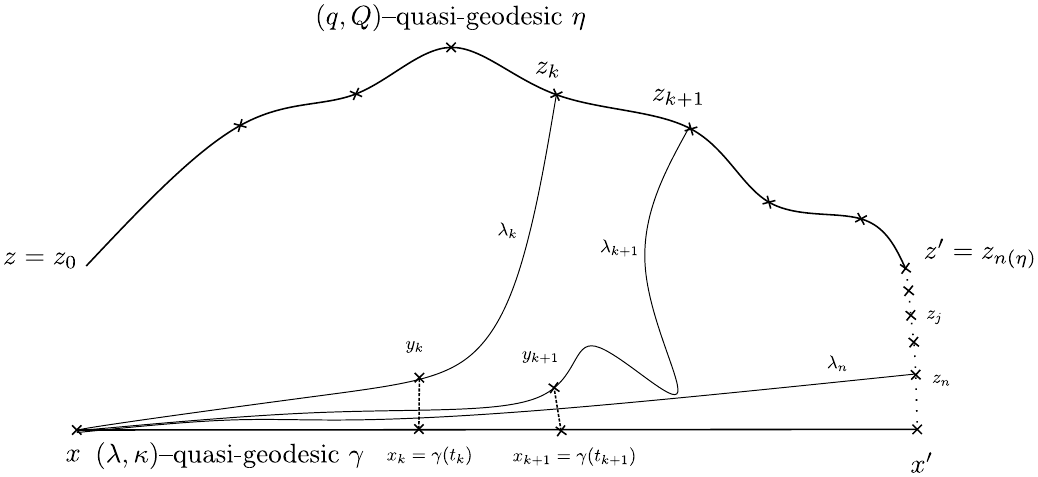}
    \caption{Setup for proof of Theorem \ref{thm:mltg}}
    \label{fig:setup}
\end{figure}

    \begin{claim}\label{claim:0}
    There exists a constant $c_1$ only depending on $N_1, \ka, \la, \ell, D$ (but not $N_2$) such that $t_0\leq c_1$ and $t_\nto\geq T - c_1$.
    \end{claim}
    \textit{Proof of Claim \ref{claim:0}:} We have that $\dist (x, z_0)= N_1$. Thus $\dlength {\col_0}\leq \la N_1+\la\ka \leq 2\la N_1$ and hence $\dist (x, \col_0(t))\leq \la 2\la N_1+\ka \leq 3\la^2 N_1$ for all $t$. Consequently, $\dist (\gamma(t_0), x)\leq D + 3\la^2 N_1\leq 4\la^2 N_1$, which implies that $t_0\leq 4\la^3 N_1 + \la\ka = c_1$. Here we used that $N_1$ is large compared to $\la, \ka$ and $D$.

    Let $\col_{\nto+1}$ be the combing line from $x$ to $x'$. By Theorem~\ref{thm:local-to-global-qg}, there is $\mu'  =\mu'(\la, \ka,\lao,\kao, M)$ such that the Hausdorff distance between $\gamma$ and $\col_{\nto +1}$ is at most $\mu'$. By boundedness the Hausdorff distance between $\col_{\nto}$ and $\col_{\nto+1}$ is at most $\lao(D+1)+\kao$. Thus, the Hausdorff distance between $\gamma $ and $\col_\nto$ is at most $\mu = \mu' + \lao(D+1)+\kao$. Let $\delta$ be the linear function of Lemma~\ref{lemma:quasi-geodesics_stay}, and fix $\ell_1 \geq \delta (\mu)$ so that if $\beta$ is a $(\la, \ka)$--quasi-geodesic, we have $d(\beta(0), \beta(\ell_1)) \geq 5\mu$. We show that we must have $T- t_n< \ell_1$. Suppose this is not the case, and let $y\in \col_\nto$ be a point closest to $\gamma(t_\nto + \ell_1)$. We have $d(y, \gamma(t_\nto+\ell_1))\leq \mu$ and $d(y_\nto, \gamma(t_\nto))\leq D \leq \mu$. By Lemma~\ref{lemma:quasi-geodesics_stay}, we must have $y\in \col_\nto\vert_{y_\nto z_\nto}$. Since $d(\gamma(t_\nto), \gamma(t_\nto+\ell))\geq 5\mu$, the triangle inequality yields $d(y_\nto, y)\geq 3\mu$. Then Lemma~\ref{lem:_wiggling_ends_of_QG} yields that the concatenation $[\gamma(t_\nto), y_{\nto}]\ast \col_{\nto}\vert _{y_{\nto} y}\ast [y, \gamma(t_{\nto} + \ell_1)]$ is a $(3\lao + 3, \kao)$--quasi-geodesic. But as the concatenation exits the $D\geq M(3\lao + 3, \kao) +1$ neighbourhood of $\gamma$, we obtain a contradiction. \hfill $\blacksquare$

    Since $\eta$ is a $(Q,q)$--quasi-geodesic, by requiring that $N_2$ is large compared to $N_1$ (and hence large compared to $\ka,\la$, $ \ell$ and $(Q,q)$) we can guarantee that $\net$ is large. This has important applications. Firstly, we note that we can make $2c_1/\nto$ and $(\nto - \net)/\net$ as small as we like, in particular, we can assume that

    \begin{align}
        2c_1/\nto &< \ell/3,\label{eq_net1}\\ 
        \nto&\leq \frac{3\net}{2}.\label{eq_net2}
    \end{align}
Moreover, control over $N_2$ allows us to bound $\dist (z_i, z_{i+1})$. For $i\geq \net$ we have $\dist (z_i, z_{i+1}) =1$, so we proceed to bound $\dist (z_i, z_{i+1})$ for $i<\net$. To do this, we first want to bound $T'$ i.e. the length of the domain of $\eta \colon [0, T'] \to X$. Observe that $T\leq 2(\net+1) \ell$ and hence $\dist (x,x')\leq 2(\net+1) \ell\la + \ka$. Therefore, by the triangle inequality,
    \begin{align*}
        \dist (z_0, z_n)\leq 2(\net+1) \ell\la + \ka + 2N_1.
    \end{align*}
    Since $\eta$ is a $(Q, q)$--quasi-geodesic we have that
    \begin{align*}
        T'\leq Q(2(\net+1) \ell\la + \ka + 2N_1) + Qq.
    \end{align*}
    Recall that for $N_2$ large enough, $N_1/\net$ can be as small as we want. Hence
    \begin{align}
        \frac{T'}{\net}&\leq Q(2 \ell\la) +  \frac{ Q(2 \ell\la + \ka + 2N_1)+q}{\net}
        \leq  Q(2 \ell\la) +  1.\label{eq:final_N2}
    \end{align}
    Lastly we can see that
    \begin{align}\label{eq:boundonzizi1}
        \dist (z_i, z_{i+1}) = d\left(\eta\left(i \frac{T'}{\net}\right), \eta\left(\left(i+1\right)\frac{T'}{\net}\right)\right) \leq   Q( Q(2 \ell\la) +  1) + q = \C,
    \end{align}
    which gives the desired bound on $\dist (z_i, z_{i+1})$. It is important that $C$ does not depend on $N_1$. More precisely, \begin{align}\label{N_i_dependency}
        \text{for each }N_1\text{ there exists }N_2\text{ that makes~\eqref{eq_net1}, \eqref{eq_net2} and \eqref{eq:boundonzizi1} true.} \end{align}
    
    \begin{claim}\label{claim:1}
        There exists $0\leq i\leq \nto-1$ such that $t_{i+1} - t_i >  \ell$. 
    \end{claim}
   
    \textit{Proof of Claim \ref{claim:1}.}
       Claim \ref{claim:0} implies that $\sum_{i=0}^{\nto-1} t_{i+1} - t_i \geq T - 2c_1$. Thus, there exists $0 \leq i\leq \nto-1$ such that $t_{i+1}- t_i \geq T/\nto - 2c_1/\nto$. By \eqref{eq_net1} and \eqref{eq_net2}, 
       \begin{align*}
           \frac{T}{\nto} - \frac{2c_1}{\nto}> \frac{2T}{3\net} - \frac{ \ell}{3} \geq \frac{4 \ell}{3} - \frac{ \ell}{3}\geq  \ell.
       \end{align*}
       \hfill$\blacksquare$

    Let $0\leq i \leq \nto-1$ be such that $t_{i+1} - t_i >  \ell$. Recall that $\delta$, respectively $\epsilon$, are the linear functions obtained by applying Lemma \ref{lemma:quasi-geodesics_stay}, respectively Lemma~\ref{lemma:nice_behavior_before_exit}, to the quasi-geodesic constants $(\la, \ka)$. Similarly, recall that $\sigma(x) = \epsilon(2x) +x$ and define $\gr = \gamma \vert_{[t_i + \delta(D), t_i +  \ell - \delta(D)]}$. In what follows, we shorten the expression \say{a path has distance at least $D$ from $\gamma_R$} with \say{a path is \emph{$D$--far}}.

    Recall that $y_i=\col_i(s_i)$ is a minimal $(D, \ell)$--exit point for $(\col_i, \gamma)$. Hence we have that $\dist (\col_i(s_i), \gamma(t_i))\leq D$ and that the following two paths are $D$--far.
    \begin{enumerate}[label=\Roman*)]
        \item[$\pgot_{\mathrm{I}}$] $:={\col_i}\vert_{[s_i, T_i]}$,\label{1}
        \item[$\pgot_{\mathrm{II}}$] $:=[\gamma(t_i), \col_i(s_i)]$.\label{3}
    \end{enumerate}
    To obtain this for $\pgot_{\mathrm{II}}$ we used Lemma~\ref{lemma:quasi-geodesics_stay}. We now want to find points $\gamma(t')$ and $\col_{i+1}(s')$ with similar properties and where in addition $t' -t_i$ is bounded from above and below in terms of $\ell$.

    \begin{claim}\label{claim:2} 
        There exist $t'\in [0, T]$ and $s'\in [0, T_{i+1}]$ such that $\ell\leq t' - t_i\leq \sigma(\ell)$ and $\dist (\gamma(t'), \col_{i+1}(s'))\leq D$. Moreover, we can choose $t'$ and $s'$ such that the following paths are $D$--far.
            \begin{enumerate}[resume, label = \Roman*)]
            \item[$\pgot_{\mathrm{III}}$] $:={\col_{i+1}}\vert_{[s', T_{i+1}]}$,\label{2}
            \item[$\pgot_{\mathrm{IV}}$] $:=[\gamma(t'), \col_{i+1}(s')]$.\label{4}
        \end{enumerate}
    \end{claim}
\textit{Proof of Claim \ref{claim:2}.}

\textbf{Case 1: $t_{i+1}-t_i \leq \sigma(\ell)$.} In this case, choose $t' = t_{i+1}$ and $s' = s_{i+1}$. The fact that $\pgot_{\mathrm{III}}$ is $D$--far holds since $\qgot_{i+1}(s_{i+1})$ is a $(D, \ell)$--exit point of $(\qgot_{i+1}, \gamma)$. Furthermore, we have that $\dist (\gamma(t'), \col_{i+1}(s'))\leq D$. The fact that $\pgot_{\mathrm{IV}}$ is $D$--far follows from Lemma \ref{lemma:quasi-geodesics_stay}.

\textbf{Case 2: $t_{i+1} - t_i > \sigma(\ell)$}.  In this case, let $t_i + \sigma(\ell) -  \ell\leq  t' \leq t_i + \sigma( \ell)$ be such that $\dist (\gamma(t'), \col_{i+1})\leq D$. Such a point must exist by definition of $t_{i+1}$. Choose $s'$ such that $\dist (\col_{i+1}(s'), \gamma(t'))\leq D$. The fact that  and $\pgot_{\mathrm{IV}}$ is $D$--far follows from  Lemma~\ref{lemma:nice_behavior_before_exit} and the choice of $\sigma$. The fact that $\pgot_{\mathrm{III}}$ is $D$--far follows from Lemma \ref{lemma:quasi-geodesics_stay} and $\dist (\col_{i+1}(s'), \gamma(t'))\leq D$.  \hfill $\blacksquare$

Define $a = \gamma(t_i), b = \gamma(t'), c = \col_{i}(s_i)$ and $d = \col_{i+1}(s')$. Now the goal is to find a path of controlled length that connects $c$ with $d$ and is $D$--far.
The idea is to \say{flow-up} the combing lines to get far away from $\gr$ and then \say{jump} from one combing line to the other using the boundedness assumption to guarantee that the two combing lines are sufficiently close to each other. \\

\textbf{Defining the path $\pgot$.} Observe that $\dist (a, u) \leq \la\sigma( \ell) + \ka$ for all $u \in \gamma\vert_{ab}$. Define $\W = 2D + \la\sigma( \ell) + \ka$, and note that $\W$ does not depend on $N_1, N_2$. By the triangular inequality, $\dist (c,d)\leq \W$ and $\dist (c, u) \leq \W$ for all $u\in \gr$. Recall that $c = \col_i(s_i), d = \col_{i+1}(s')$ and that $\C$ is a bound on $\dist (z_i, z_{i+1})$ not depending on $N_1, N_2$. Define 
\begin{align}\label{eq:weakh}
    \H = \la(\ka \C + 2\ka + D + \W). 
\end{align}
Now, choose $N_1$ large enough so that the domain length of $\col_i$ is at least $s_i + \H$, and choose $N_2$ large enough for \ref{N_i_dependency}. Hence we have
\begin{align}\label{eq:weak:dist_to_H}
    \dist (\col_i(\H + s_i), \col_i(s_i))\geq \ka \C + \ka + D + \W.
\end{align}
By boundedness, there exist $s$ such that 
\begin{align}\label{eq:weak:s}
    \dist (\col_i(\H + s_i), \col_{i+1}(s))\leq \ka\C + \ka.
\end{align}

Consider the following path 
\begin{align}
   \pgot = [a, c]\ast \col_i\vert_{[s_i, s_i+\H]}\ast [\col_i(s_i+\H), \col_{i+1}(s)]\ast \col_{i+1}\vert_{[s, s']} \ast [d, b].
\end{align}

\begin{claim}\label{claim:6}
The path $\pgot$ satisfies the following two properties: 
\begin{enumerate}
    \item $\pgot$ is $D$--far;\label{prop:weak-distance}
    \item $\length {\pgot}\leq \K(\ell)$, where $\K$ is a linear function that can be defined only in terms of $(\la, \ka, \lao, \kao, Q, q)$. \label{prop:weak_length_bound}
\end{enumerate}
\end{claim}
\textit{Proof of Claim \ref{claim:6}.} (1): By the triangle inequality, \eqref{eq:weak:dist_to_H} and \eqref{eq:weak:s}, we have that $\dist ([\col_i(s_i+\H(\ell)), \col_{i+1}(s)], \gr)\geq D$. Hence $\dist (\pgot, \gr)\geq D$ follows from the fact that $\pgot_{\mathrm{I}}, \dots, \pgot_{\mathrm{IV}}$ are $D$--far.

(2): Since $\pgot$ is a concatenation of five $(\la, \ka)$--quasi-geodesics, by Lemma~\ref{lemma:tamingqg}, it suffices to show that for each of them the distance of their endpoints is linearly bounded in $\ell$. The length of $[a, c], [b, d]$ is at most $D\leq \ell$ by construction. For $\col_i\vert_{[s_i, s_i+\H]}$ it follows since $\H$ is bilinear in $(D,\ell)$ and hence has a linear upper bound in $\ell$, and $\H$ is defined in terms of $(\la, \ka, \lao, \kao,Q, q)$. For $[\col_i(s_i+\H(\ell)), \col_{i+1}(s)]$ it follows from \eqref{eq:weak:s}, where again $C$ is linear in $\ell$. Lastly, for $\col_{i+1}\vert_{[s, s']}$ it follows from the triangle inequality, because both $\dist (a, b)$ and all of the other four subsegments had endpoints at linearly bounded distance. Thus there exists a linear function $\K$ such that $\length{\pgot}\leq \K(\ell)$. \hfill$\blacksquare$.\\

\textbf{Concluding the proof.} We now have a path $\pgot$, which by Lemma~\ref{lemma:mid-recurrence-consequence'} cannot exist, a contradiction. Hence, all $(Q, q)$--quasi-geodesic $\eta$ with endpoints in the $N_1$--neighbourhood of $\gamma$ stay in the $N_2$--neighbourhood. Implying any $(\la, \ka)$--quasi-geodesic which is $L$--locally $M$--Morse is $(N_2, Q, q)$--weakly Morse. The scale $L$ depends on $(Q, q)$. 
\end{proof}

\section{The Weak Morse local-to-global and sigma-compactness of the Morse boundary}\label{sec:WMLTG}

The goal of this section is to prove Theorem~\ref{thmi:non-sigma} which states that a geodesic space (whose isometry groups acts coboundedly on itself) with the weak Morse MLTG property has either the (ordinary) MLTG property or has non-$\sigma$-compact Morse boundary. We remind that a topological space is \emph{$\sigma$--compact} if it can be written as a countable union of compact sets. Let us fix some notation.

\begin{definition}[Exhaustion]
    We say that a sequence $(M_n)_{n\in \N}$ is an \emph{exhaustion} of $\mb X$ if $M_n\leq M_{n+1}$ for all $n$ and for all Morse rays $\gamma: [0, \infty)\to X$ starting at $\bp$ we have that $\gamma$ is $M_n$--Morse for some $n$.
\end{definition}

Observe that the Morse boundary $\mb X$ is $\sigma$ compact if and only if there exists an exhaustion of $\mb X$ by \cite[Lemma 2.1]{Cordes:Morse}, \cite[Lemma 4.1]{cordesdurham:boundary} and \cite{CordesSistoZbinden:corrigendum}. We want to recall the following well-known facts about Morse quasi-geodesics references for proofs can be found for example in \cite[Lemma~2.8]{zbinden:morse}.

\begin{lemma}\label{lemma:monster}
    Let $M$ be a Morse gauge and $(\la, \ka)$--quasi-geodesic constants, there exists a Morse gauge $M'$ only depending on $M, \la$ and $\ka$ such that the following holds: 
    \begin{enumerate}
        \item \textit(Triangles) Let $\Delta$ be a triangle with (potentially unbounded) sides $\alpha, \beta, \gamma$ which are $(\la, \ka)$--quasi-geodesics. If $\alpha$ and $\beta$ are $M$--Morse, then $\gamma$ is $M'$--Morse. \label{prop:triangles}
        \item \textit (Equivalent quasi-geodesics) Let $\alpha, \beta$ be $(\la, \ka)$--quasi-geodesics which are at bounded Hausdorff distance and which have the same starting point. If $\alpha$ is $M$--Morse, then $\beta$ is $M'$--Morse.\label{prop:equivalence}
        \item \textit{ (Concatenation)} Let $\alpha$ and $\beta$ be $M$--Morse $(\la, \ka)$--quasi-geodesics such that their concatenation $\gamma = \alpha\ast \beta$ is a $(\la, \ka)$--quasi-geodesic. Then $\gamma$ is $M'$--Morse.\label{prop:concatentaion}
    \end{enumerate}
\end{lemma}

\subsection{Properties of locally Morse geodesics}
We start by generalizing some basic properties of Morse geodesics for locally Morse geodesics. We only need this in a very specific setting, and hence the statement and the proofs are specialized accordingly. However, the proof can be easily adapted to other, more general, settings.

\begin{lemma}\label{lemma:stiching-quasi-geodesics1}
Let $X$ be a geodesic metric space. Let $M, N$ be Morse gauges and let $(\la, \ka)$ be quasi-geodesic constants. There exists a Morse gauge $N'$, a constant $\delta$ and a function $\tau: \R_{\geq 0}\to \R_{\geq 0}$ such that for every constant $L$ the following holds. Let $\eta$ be an $N$--Morse geodesic segment with endpoints $x, z$, and let $\gamma$ be a $(\la, \ka)$--quasi-geodesic segment which is $\tau(L)$--locally $M$--Morse with endpoints $u$ and $v$. If $\dist(u, z)\leq r$, then for any $y\in \eta$ the geodesic $[y, v]$ is $L$-locally $N'$--Morse and contained in the $\delta$--neighbourhood of $\eta \vert_{yz}\cup \gamma$.
\end{lemma}

\begin{figure}[ht]
    \centering
    \includegraphics{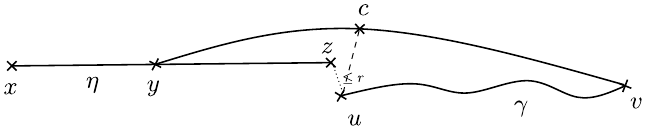}
    \caption{Proof of Lemma \ref{lemma:stiching-quasi-geodesics1}}
    \label{fig:stiching1}
\end{figure}

\begin{proof}
Let $c$ be a closest point to $u$ on $[y,v]$. By Lemma~\ref{lem:_wiggling_ends_of_QG}, $\pgot_1 = [u, c]\ast [y, v] \vert_{cv}$ is a $(3, 0)$--quasi-geodesic. By the weak MLTG property, up to choosing $\tau(L)$ large enough, we can guarantee that $\pgot_1$ lies in a uniform neighbourood of $\gamma$. In particular, it is $L$--locally Morse, where the constants can be determined in terms of $L$ and the local Morse gauge of $\gamma$. Now consider the concatenation $\pgot_2 = [y, v] \vert_{yc}\ast [c, u] \ast[u, z]$. It is a $(3, 2r)$--quasi-geodesic. Since $\eta$ is Morse, we have that $\pgot_2$ lies in a uniform neighbourhood of $\eta$ and it is Morse, where the Morse gauge depends only on $r$ and $N$. Thus, $[y, v]$ is the concatenation of two geodesics which are both $L$--locally Morse. In particular, $[y,v]$ is $L$--locally Morse, for a worse Morse gauge. Moreover, since $\pgot_1$ and $\pgot_2$ are both in a uniform neighbourhood of $\eta \vert_{yz}\cup \gamma$, so is $[y, v]$.
\end{proof}

\begin{lemma}\label{lemma:stiching-quasi-geodesics2}
Let $X$ be a geodesic metric space. Let $M, N$ be Morse gauges and let $(\la,\ka)$ be quasi-geodesic constants. There exist Morse gauges $N'$ and $ N_{\mathrm{not}}$, functions $f:\R_{\geq 0 }\times \R_{\geq 0}\to \R_{\geq0}$ and $g : \R_{\geq 0}\to \R_{\geq 0}$ and quasi-geodesic constants $(\la', \ka')$ such that for every constant $L$ the following holds. For $i=1, 2$, let $\eta_i$ be an $N$--Morse geodesic segment with endpoints $x_i, z_i$, and let $y_i$ be a point on $\eta_i$. Let $\gamma$ be a $(\la, \ka)$--quasi-geodesic segment which is $g(L)$--locally $M$--Morse but not $N_{\mathrm{not}}$--Morse and whose endpoints $u$ and $v$ satisfy
\begin{align*}
  \dist(z_1, u)\leq r &\quad\text{and}\quad \dist(v, x_2)\leq r,\\
  \dist(y_1, z_1) \geq f(\dist (u, v), L) &\quad\text{and} \quad \dist (x_2, y_2)\geq f(\dist (u, v), L).   
\end{align*}
Then the path $\pgot = \eta_1 \vert_{x_1y_1}\ast[y_1, y_2]\ast\eta_2 \vert_{ y_2 z_2}$ is a $(\la', \ka')$--quasi-geodesic which is $L$-locally $N'$--Morse. Moreover, if $\pgot$ is an $\tilde{N}$--Morse $(\la'', \ka'')$--quasi-geodesic, then $\gamma$ is $\tilde{N}'$--Morse, where $\tilde{N}'$ only depends on $\tilde{N}, N, \la, \ka, r$ and $(\la'', \ka'')$. 
\end{lemma}

\begin{figure}[ht]
    \centering
    \includegraphics[width=\linewidth]{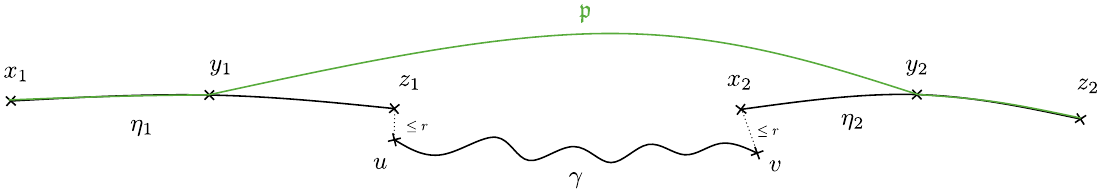}
    \caption{Depiction of Lemma \ref{lemma:stiching-quasi-geodesics2}}
    \label{fig:stiching2}
\end{figure}

\begin{proof}
    Applying Lemma \ref{lemma:stiching-quasi-geodesics1} first to $[y_1, v]$ and then to $[y_1, y_2]$ yields the existence of a function $g$, a Morse gauge $N''$ and a constant $\delta$ such that if $\gamma$ is $g(L)$--locally $M$--Morse, then $[y_1, y_2]$ is $L$-locally $N''$--Morse and $[y_1, y_2]$ is contained in the $\delta$--neighbourhood of $\eta_1 \vert_{y_1 z_1}\cup \gamma \cup \eta_2 \vert_{x_2 y_2}$. 

    If $\dist(\eta_1, \eta_2)\leq 2\delta$, then applying Lemma \ref{lemma:monster}\ref{prop:triangles} about Morse triangles multiple times yields a Morse gauge $N_0$ depending only on $N, \la, \ka, r$ and $\delta$ such that $\gamma$ is $N_0$--Morse. Defining $N_{\mathrm{not}} = N_0$ yields that $\dist(\eta_1, \eta_2)> 2\delta$.

    Next we show that $\pgot' = \eta_1\vert_{x_1y_1}\ast[y_1, y_2]$ is $L$--locally an $N'$--Morse $(\la', \ka')$--quasi-geodesic for a Morse gauge $N'$ and quasi-geodesic constants $(\la', \ka')$ which we will determine below. Let $y$ be the first point on $[y_1, y_2]$ which is in the closed $\delta$--neighbourhood of $\gamma \cup \eta_2$. By continuity, $y$ is in the closed $\delta$--neighbourhood of $\eta_1 \vert_{y_1 z_1}$. Since $\dist(\eta_1, \eta_2)>2\delta$, $y$ is not in the closed $\delta$--neighbourhood of $\eta_2$ and hence in the closed $\delta$--neighbourhood of $\gamma$. We now proceed to bound $\dist(y, y_1)$ from below. By the triangle inequality, we have that 
    \begin{align*}
        \dist(y, y_1) &\geq \dist(y_1, \gamma) - \delta,\\
                &\geq \dist(y_1, z_1) - \la(\la\dist(u, v) + \ka) - \ka - \delta - r.
    \end{align*}
    Thus, for large enough $f$, we have that $\dist (y, y_1)\geq L$. Define $(\la', \ka') = (1, 2\delta)$. To prove that $\pgot'$ is $L$-locally a $(\la', \ka')$--quasi-geodesic it suffices to show that $\dist (\pgot'(s), \pgot'(t))\geq \vert s  - t\vert - 2\delta$ for all $s, t$ with $\pgot'(s)\in \eta_1 \vert_{x_1 y_1}$ and $\pgot'(t)\in [y_1, y_2]$ with $\dist(\pgot'(t), y_1)\leq L$. Let $s, t$ be such constants. Since $\dist(y, y_1)\geq L$, there exists $c$  on $\eta_1 \vert_{y_1 z_1}$ with $\dist(\pgot'(t), c)\leq \delta$ and hence $\dist(y_1, c)\geq \dist(y_1, \pgot'(t)) - \delta$. Observe that $\vert s - t \vert = \dist(\pgot'(s), y_1) + \dist (y_1, \pgot'(t))$, implying that $\dist(\pgot'(s),\pgot'(t)) \geq \dist(\pgot'(s), c) - \delta \geq \vert s - t \vert -2\delta$. Hence $\pgot'$ is indeed $L$--locally a $(\la', \ka')$--quasi-geodesic. Furthermore, $[y_1, y_2]$ is $L$-locally $N''$--Morse and $\eta_1 \vert _{x_1 y_1}$ is $N$--Morse. Thus, by Lemma \ref{lemma:monster}\eqref{prop:concatentaion} about concatenations there exists a Morse gauge $N'$ only depending on $N, N'', \la$ and $\ka$ such that $\pgot'$ is $L$-locally $N'$--Morse.

    Analogously, we can show that $[y_1, y_2]\ast \eta_2 \vert _{y_2 z_2}$ is $L$--locally an $N'$--Morse $(\la', \ka')$--quasi-geodesic. Lastly, since $\dist(y_1, y_2)\geq L$ this shows that $\pgot$ as a whole is $L$--locally an $N'$--Morse $(\la', \ka')$--quasi-geodesic.

    It remains to prove the moreover part of the statement, which follows from repeatedly applying Lemma \ref{lemma:monster}\eqref{prop:triangles} about Morse triangles.
\end{proof}

\subsection{Weak Morse local-to-global} We use the above lemmas to construct specific concatenations in spaces with the weak Morse local-to-global property. This is the key lemma in the proof of Theorem~\ref{thmi:non-sigma}.

\begin{lemma}\label{lemma:stiching-quasi-geodesics3}
    Let $X$ be a geodesic metric space which satisfies the weak MLTG property and whose isometry group acts coboundedly on itself. Let $M$ and $N$ be Morse gauges and let $(\la, \ka)$ be quasi-geodesic constants. There exists a Morse gauge $N_{\mathrm{not}}$, a constant $L_{\mathrm{min}}$, and a map $\Phi$ between Morse gauges such that the following holds. Suppose there exists an $N$--Morse geodesic ray $\eta$. Then for each sequence $\{\gamma_i\}_{i\in \N}$ of $(\la, \ka)$--quasi-geodesic segments which are $L_i$-locally $M$--Morse but not $N_{\mathrm{not}}$--Morse and with $L_i\geq L_{\mathrm{min}}$ and $\lim_{i\to\infty} L_i = \infty$, there exists a Morse geodesic ray $\zeta$ such that for all Morse gauges $\tilde{N}$, if $\zeta$ is $\tilde{N}$--Morse, then $\gamma_i$ is $\Phi(\tilde{N})$--Morse for all $i$.
\end{lemma} 
\begin{proof}
Let $N', N_{\mathrm{not}}, (\la', \ka')$ and $f,g$ be the Morse gauges, quasi-geodesic constants and functions from Lemma \ref{lemma:stiching-quasi-geodesics2} applied to the Morse gauges $M, N$ and quasi-geodesic constants $(\la, \ka)$. By potentially increasing $N'$, we may assume that $N'\geq N$. Let $L_{\mathrm{quasi}}, \la'', \ka''$ be constants such that every $L_{\mathrm{quasi}}$-locally $N'$--Morse $(\la', \ka')$--quasi-geodesic is a $(\la'', \ka'')$--quasi-geodesic. Let $L_{\mathrm{min}} = g(L_{\mathrm{quasi}})$. For every quasi-geodesic pair $(Q, q)$, let $L_{Q, q}$ and $M(Q, q)$ be constants such that every $L_{Q,q}$--locally $N'$--Morse $(\la', \ka')$--quasi-geodesic is a $(M(Q, q), Q, q)$--Morse $(\la'', \ka'')$--quasi-geodesic.

Let $\{\gamma_i\}_{i\in \N}$ be a sequence as in the statement. For every $i\geq 1$ let $D_i = f(\dist (u_i, v_i), L_i)$, where $u_i$ and $v_i$ are the endpoints of $\gamma_i$.

    Define $\eta_1' = \eta$ and let $x_1 =y_1$ be the starting point of $\eta_1'$. Let $z_1$ be a point on $\eta_1'$ such that $\dist (y_1, z_1)\geq D_1$. Define $\eta_1$ as $\eta_1' \vert_{y_1 z_1}$. Lastly, define $\gamma_1'$ as a translate of $\gamma_1$ with endpoints $u_1', v_1'$ such that $\dist (u_1', z_1)\leq r$. For $i\geq 1$ inductively define the following, which is depicted in Figure \ref{fig:stiching3}.

\begin{figure}[ht]
    \centering
    \includegraphics[width=\linewidth]{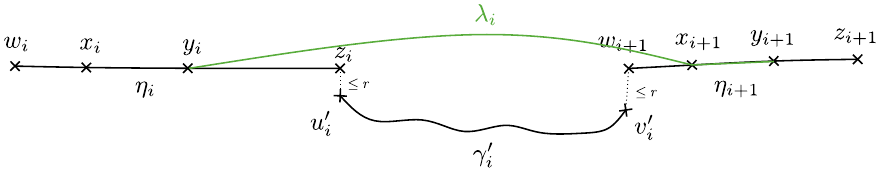}
    \caption{Proof of Lemma \ref{lemma:stiching-quasi-geodesics3}}
    \label{fig:stiching3}
\end{figure}
    \begin{itemize}
        \item $\eta_{i+1}'$ as a translate of $\eta$ starting in the $r$--neighbourhood of $v_i'$
        \item $w_{i+1}$ as the starting point of $\eta_{i+1}'$
        \item $x_{i+1}$ as a point on $\eta_{i+1}'$ with $\dist (w_{i+1}, x_{i+1}) = D_i$.
        \item $y_{i+1}$ as a point on $\eta_{i+1}'$ with $\dist (w_{i+1}, y_{i+1}) = D_i + L_{\mathrm{quasi}}+i$, and hence $\dist (x_{i+1}, y_{i+1}) = L_{\mathrm{quasi}}+i$. 
        \item $z_{i+1}$ as a point on $\eta_{i+1}'$ with $\dist (w_{i+1}, z_{i+1}) = D_i + L_{\mathrm{quasi}}+i +D_{i+1}$ and hence $\dist (y_{i+1}, z_{i+1}) = D_{i+1}$.
        \item $\eta_{i+1}$ as $\eta_{i+1}'\vert_{w_{i+1} z_{i+1}}$. 
        \item $\gamma_{i+1}'$ as a translate of $\gamma_{i+1}$ with endpoints $u_{i+1}', v_{i+1}'$ and such that $\dist (u_{i+1}', z_{i+1})\leq r$.
    \end{itemize}

For $i\geq 1$ define $\zeta_i = [y_i, x_{i+1}]\ast[x_{i+1}, y_{i+1}]$. And for each $i$ define $\zeta_i'$ as the infinite concatenation 
\begin{align*}
    \zeta_i' = \zeta_i\ast \zeta_{i+1}\ast\zeta_{i+2}\ast \ldots
\end{align*}

Next we show that $\zeta'_1$ is a quasi-geodesic. Indeed, since $L_i \geq g(L_{\mathrm{quasi}})$ for all $i$, the path $[x_i, y_i]\ast[y_i, x_{i+1}]\ast[x_{i+1}, y_{i+1}]$ is an $L_{\mathrm{quasi}}$--locally $N'$--Morse $(\la', \ka')$--quasi-geodesic by Lemma \ref{lemma:stiching-quasi-geodesics2}. Since $\dist (x_i, y_i)\geq L_{\mathrm{quasi}}$, the above implies that $\zeta_{1}'$ is $L_{\mathrm{quasi}}$--locally an $N'$--Morse $(\la', \ka')$--quasi-geodesic. By the choice of $L_{\mathrm{quasi}}$, $\zeta_{1}'$ is a $(\la'', \ka'')$--quasi-geodesic. 

Now we show that $\zeta_1'$ is $N_0$--Morse for a Morse gauge $N_0$ which we are about to construct. Let $(Q, q)$ be a quasi-geodesic pair. Since $\lim_{i\to \infty} L_i = \infty$, there exists $i_{Q, q}\geq L_{Q, q}$ such that for all $i\geq i_{Q, q}$ we have that $L_i\geq g(L_{Q, q})$. Hence by Lemma \ref{lemma:stiching-quasi-geodesics2}, $[x_i, y_i]\ast[y_i, x_{i+1}]\ast[x_{i+1}, y_{i+1}]$ is $L_{Q, q}$--locally an $N'$--Morse $(\la', \ka')$--quasi-geodesic for all $i\geq i_{Q, q}$. Since $i_{Q, q}\geq L_{Q, q}$ and $\dist (x_{i}, y_{i})\geq i$ we have that $\zeta_{i_{Q, q}}'$ is also $L_{Q, q}$--locally an $N'$--Morse $(\la', \ka')$--quasi-geodesic. Hence, by the definition of $L_{Q, q}$, $\zeta_{i_{Q, q}}'$ is $(M_{Q, q}, Q, q)$--Morse. Consider $\pgot_{Q, q} = \zeta_1\ast \ldots \ast \zeta_{i_{Q, q}-1}$. We have that $\zeta_1' = \pgot_{Q, q}\ast \zeta_{i_{Q, q}}'$. Further, since $\pgot_{Q, q}$ is a finite subsegment of $\zeta_1'$, there exists a constant $N_{Q, q}$ such that $\pgot$ is $(N_{Q, q}, Q, q)$--Morse. Define 
\begin{align*}
    N_0(Q/3, q) = \max\{N_{Q, q}, M_{Q, q}\}. 
\end{align*}

\begin{claim}\label{claim:stich}
    The quasi-geodesic $\zeta_1'$ is $N_0$--Morse.
\end{claim}

\textit{Proof of Claim \ref{claim:stich}.} 
Let $\xi$ be a $(Q/3, q)$--quasi-geodesic with endpoints $a$ and $b$ on $\zeta_1'$. If both its endpoints are on $\pgot_{Q, q}$, then, since $\pgot_{Q, q}$ is $(N_{Q, q}, Q, q)$--Morse $\xi$ is contained in the $N_{Q, q}$--neighbourhood of $\zeta_1'\vert_{a b}$. If both endpoints $a$ and $b$ are contained in $\zeta_{i_{Q, q}}'$, then, since $\zeta_{i_{Q, q}}'$ is $(M_{Q, q}, Q, q)$--Morse, $\xi$ is contained in the $M_{Q, q}$ neighbourhood of $\zeta_1' \vert_{a b}$. It remains to show that $\xi$ stays in the $N_0 (Q/3, q)$--neighbourhood of $\zeta_1 \vert _{[a, b]}$ if $a$ lies on $\pgot_{Q, q}$ and $b$ lies on $\zeta_{i_{Q, q}}'$. Let $c$ be the endpoint of $\pgot_{Q, q}$ (and hence the starting point of $\zeta_{i_{Q, q}}'$). Let $c'$ be a closest point on $\xi$ to $c$. By Lemma~\ref{lem:_wiggling_ends_of_QG} the paths $\xi \vert _{a c'}\ast [c', c]$ and $[c, c'] \ast\xi \vert _{c' b}$ are $(Q, q)$--quasi-geodesics. By the arguments above, they are contained in the $N_{Q, q}$ and $M_{Q, q}$--neighbourhood of $\zeta_1' \vert_{a c}$ and $\zeta_1' \vert_{c b}$ respectively. Since $N_0(Q/3, q) = \max\{N_{Q, q}, M_{Q, q}\}$, the statement follows \hfill$\blacksquare$

We have showed that $\zeta_1'$ is a Morse $(\la'', \ka'')$--quasi-geodesic, hence there exists a geodesic $\zeta$ with the same starting point which is at bounded Hausdorff distance from $\zeta_1'$. Note that $(\la'', \ka'')$ only depend on $X, M, \la, \ka$ and not the quasi-geodesics $\gamma_i$. Assume that $\zeta$ is $\tilde{N}$--Morse for some Morse gauge $\tilde{N}$. By Lemma~\ref{lemma:monster}\eqref{prop:equivalence}, $\zeta_1'$ is $\tilde{N}''$--Morse, where $\tilde{N}''$ only depends on $\tilde{N}, \la''$ and $\ka''$. Since $[x_i, y_i]\ast[y_i, x_{i+1}]\ast[x_{i+1}, y_i]$ is a subsegment of $\zeta_1'$, it is also $\tilde{N}''$--Morse. By Lemma~\ref{lemma:stiching-quasi-geodesics2}, $\gamma_i$ is $\tilde{N}'$--Morse, where $\tilde{N'}$ only depends on $\tilde{N}'', N, M, X, \la$ and $\ka$, which concludes the proof.
\end{proof}

Now we are ready to prove Theorem \ref{thmi:non-sigma}.

\begin{proof}[Proof of Theorem \ref{thmi:non-sigma}]
We assume that $X$ satisfies the weak MLGT property but not the MLTG property. Observe that any space with empty Morse boundary satisfies the MLTG, thus $\mb X$ is not empty, and there exists a ray $\eta$ which is $N$--Morse for some Morse gauge $N$. We want to show that $\mb X$ is not $\sigma$-compact. More precisely, we will prove that $X$ does not have an exhaustion.

Assume by contradiction that $X$ has an exhaustion $(M_n)_{n\in \N}$. Let $(M, \la', \ka')$ be a triple that fails the Morse local-to-global property. Since $X$ satisfies the weak MLTG, there exist constants $L_0, \la, \ka$ such that all $L_0$--locally $M$--Morse $(\la, \ka')$--quasi-geodesics are $(\la, \ka)$--quasi-geodesics. Thus we can apply Lemma \ref{lemma:stiching-quasi-geodesics3} to the Morse gauges $M$ and $N$ and the quasi-geodesic constants $(\la, \ka)$ to get a Morse gauge $N_\mathrm{not}$, constant $L_{\mathrm{min}}$ and a map $\Phi$ between Morse gauges.

For each $i\geq 1$, define $L_i = \max\{i, L_{\mathrm{min}}, L_0\}$. Since $(M, \la', \ka')$ fails the Morse local-to-global property, there exists a path $\gamma_i$ which is $L_i$--locally an $M$--Morse $(\la', \ka')$-quasi-geodesic but not a $\max\{N_{\mathrm{not}}, \Phi(M_i)\}$--Morse $(\la, \ka)$--quasi-geodesic. Since $L_i \geq L_0$,  $\gamma_i$ is a $(\la, \ka)$--quasi-geodesic and hence the failure of the MLTG property is that $\gamma_i$ is not $\Phi(M_i)$--Morse. Moreover, by potentially replacing $\gamma_i$ with a subsegment, we can assume that $\gamma_i$ is finite. Indeed, if all finite subsegments of $\gamma_i$ were $\Phi(M_i)$--Morse, $\gamma_i$ itself would be $\Phi(M_i)$--Morse.

Let $\zeta$ be the Morse geodesic obtained from Lemma \ref{lemma:stiching-quasi-geodesics3} applied to the sequence $\{\gamma_i\}$. Since $\zeta$ is Morse and $(M_n)_{n\in\N}$ is an exhaustion, there exists a Morse gauge $M_n$ such that $\zeta$ is $M_n$--Morse, implying that $\gamma_n$ is $\Phi(M_n)$--Morse. However, we precisely chose $\gamma_n$ to not be $\Phi(M_n)$--Morse, which is a contradiction. The statement follows.
\end{proof}

\appendix

\section{A survey on the MLTG property}\label{sec:survey}
One of the current programs in geometric group theory is to understand groups that are not Gromov-hyperbolic, but that present \say{negative curvature features}. The precise meaning of \say{negative curvature features} gave rise to an archipelago of different definitions, and the past years saw incredible progress in the difficult task of mapping and understanding the relations and limitations of the various notions of negative curvature outside hyperbolic groups. 

The main notion we are concerned with is the notion of \emph{Morse} quasi-geodesics (Definition~\ref{defn:Morse_QG}). The slogan is that  a quasi-geodesic is Morse if it behaves like a quasi-geodesic in a hyperbolic group. This is justified by a theorem of Bonk (\cite{Bonk:quasigeodesic}) stating that a geodesic metric space is hyperbolic if and only if there exists a Morse gauge $M$ so that all of its geodesics are $M$--Morse. The systematic study of Morse quasi-geodesic was formalized in \cite{Cordes:Morse}, and became a very popular topic, allowing to refine the quasi-isometric classification of groups and subgroups \cite{OlshankiiOsinSapir:lacunary}, \cite{DurhamTaylor:convexcocompact}, \cite{aougabdurhamtaylor:pulling}, \cite{CharneyCordesSisto:complete}, \cite{zbinden:morse}, \cite{fioravantikarrersistozbinde:cech}. Indeed, a key property of Morse quasi-geodesics is that they are preserved by quasi-isometric embeddings: if $\gamma \colon I \to X $ is a Morse quasi-geodesic and $f \colon X \to Y$ is a quasi-isometric embedding, then $f \circ \gamma$ is a Morse quasi-geodesic of $Y$. Thus, if two metric spaces are quasi-isometric, their respective sets of Morse quasi-geodesics should \say{display the same pattern}. This is made formal by \cite[Main theorem]{Cordes:Morse} stating that two quasi-isometric spaces have homeomorphic Morse boundaries. 

Since Morse quasi-geodesics behave as quasi-geodesics in hyperbolic spaces, it is natural to expect that theorems about geodesics in hyperbolic spaces should naturally generalize to Morse quasi-geodesics in general geodesic metric spaces. However, \textbf{this turns out not to be true}. For instance, in \cite{OOS}, the authors produced examples of finitely generated, non-cyclic, torsion-free groups, where every group element is Morse, meaning that for every $g\in G$ the orbit map $\mathbb{N}\to g^n \cdot x$ is a Morse quasi-geodesic, but the group itself does not have free subgroups. This tends to complicate things. A basic property of hyperbolic groups is that one can apply the ping-pong lemma: if two elements are not contained in the same cyclic group, appropriate powers of them generate a free group. However, in the group of \cite{OOS} one cannot do so. Note that the ping-pong lemma only requires to consider two elements, so they will both be Morse for the same Morse gauge. To conclude, although Morse quasi-geodesics are very similar to quasi-geodesics in  hyperbolic spaces, \textbf{even very basic proofs cannot be translated to Morse quasi-geodesics in geodesic metric spaces}. 

Introduced in \cite{RussellSprianoTran:thelocal}, the Morse local-to-global property (Definition~\ref{defn:MLTG}) appears to address the above issue in a satisfactory manner, and thus we have a moral slogan: 
\[\textit{In a MLTG space, Morse quasi-geodesics behaviour is closer to the one in hyperbolic spaces.}\]
The present appendix is a survey of results on the MLTG property intended to clarify why the above slogan is reasonable, and what the main open questions and goals in the field are. Before we dive into it, let us briefly review some assets and potential drawbacks of the MLTG property.

\textbf{Assets}
\begin{enumerate}
    \item The MLTG property is satisfied by a very large class of interesting examples; up to now, there is no known finitely presented counterexample. 
    \item Many theorems require the assumption of the MLTG property, and are false without it. This makes MLTG a meaningful property.
\end{enumerate}

\textbf{Drawbacks}
\begin{enumerate}
    \item Since the property is satisfied by so many classes of groups, it can only imply properties that are satisfied by all the groups in these classes. 
    \item The MLTG property is \emph{only concerned with Morse quasi-geodesics}. It does not have an impact, \textit{a priori}, on phenomena unrelated to some form of Morse property. 
\end{enumerate}

\subsection{Examples and non-examples}

Note that many of the properties and results can be translated to quasi-geodesic metric spaces, namely spaces where every pair of points can be joined by a uniform quasi-geodesic. Following immediately from the definition we have the following basic property. 
\begin{lemma}
Let $X$ be a metric space with the MLTG property and let $Y$ be a metric space quasi-isometric to it. Then $Y$ has the MLTG property. 
\end{lemma}
In particular, the MLTG property is well-defined for finitely generated groups. 

We are now interested in which groups and spaces satisfy the MLTG property. Firstly, observe that the MLTG property has the following form: \begin{align*}    
\text{Given a Morse gauge } M\text{ and quasi-geodesic constants }(\la, \ka)\text{ there exists }L\\ \text{ so that if a path is }L\text{--locally an } M\text{--Morse }(\la, \ka)\text{--quasi-geodesic then \dots.}\end{align*}

Note that if a space is so that for every Morse gauge $M$ there is a maximal length $L_M$ so that every quasi-geodesic that is $M$--Morse has length at most $L_M$, then the above property holds trivially, as choosing $L > L_M$ makes it a statement about the empty set. This motivates the following definition.

\begin{definition}[Morse limited]
    A metric space is \emph{Morse limited} if for every pair of quasi-geodesic constants $(\la, \ka)$ and Morse gauge $M$ there exists a constant $L$ so that if $\gamma \colon [a,b]\to X$ is an $M$--Morse $(\la, \ka)$--quasi-geodesic, then $\vert a-b \vert \leq L$.
\end{definition}

We start with a collection of theorems concerned with the examples satisfying the MLTG property.
\begin{theorem}\label{appendix:MLTG_examples}
The following groups and spaces satisfy the MLTG property. 
\begin{enumerate}
    \item \label{item:Morse_limited}Morse limited spaces \cite{RussellSprianoTran:thelocal}. This includes products with unbounded factors, spaces where all asymptotic cones do not have cut-points, and finitely generated groups where a single asymptotic cone does not have cut-points. The latter include  
    \begin{enumerate}
        \item Groups that satisfy a law, such as uniformly amenable groups, solvable groups, Burnside groups \cite{DrutuSapir:TreeGraded}.
        \item Groups with an infinite order central element \cite{DrutuSapir:TreeGraded}. 
        \item certain lacunary hyperbolic groups \cite{OOS}.
    \end{enumerate}
    \item Hyperbolic groups and spaces \cite{Gromov:hyperbolic}.
    \item Hierarchically hyperbolic groups and spaces \cite{RussellSprianoTran:thelocal}, such as Mapping Class Groups, Teichm\"uller spaces \cite{BehrstockHagenSisto:hierarchically}, extra-large type Artin groups \cite{HagenMartinSisto:extra-large}, extensions of lattice Veech groups \cite{DowdallDurhamLeiningerSisto:extensionsII}. 
    \item CAT(0) groups and spaces, in particular all cubical groups \cite{RussellSprianoTran:thelocal, CharneySultan:contracting}.
    \item \label{item:Morse_Dichotomous}In general, spaces where Morse is equivalent to being \emph{strongly contracting} \cite[Proposition~4.7]{sistozalloum:morse}, such as injective groups and spaces \cite{sistozalloum:morse}, and  properly convex domains \cite{IslamWeisman:morseproperties}. 
    \item \label{item:rel_to_MLTG}Groups hyperbolic relative to groups with the MLTG property \cite{RussellSprianoTran:thelocal}. 
    \item \label{item:3_mflds}All fundamental groups of closed three manifolds \cite{RussellSprianoTran:thelocal}.     
    \item All $C'(1/9)$--small-cancellation groups with $\sigma$-compact Morse boundary \cite{HeSprianoZbinden:sigmacompactnes}.
  \item Geodesic metric spaces with a bounded combing and $\sigma$-compact Morse boundary (Theorem~\ref{thm:main_intro}).
    \item \label{item:WMLTG} Weak MLTG spaces with $\sigma$--compact Morse boundary (Theorem~\ref{thmi:WMLTG}).
\end{enumerate}
\end{theorem}
We remark that in \eqref{item:Morse_Dichotomous} it is important that the dependence is quantitative. With the terminology of \cite{Zbinden:hyperbolic}, those are Morse dichotomous spaces, and not just weakly Morse dichotomous spaces. 

As mentioned, Morse limited spaces (\ref{item:Morse_limited}) do not have Morse geodesics, and hence the MLTG property does not give us any information on them. However, they are absolutely fundamental in light of \eqref{item:rel_to_MLTG}. Indeed, it is only through the Morse limited examples that we can show \eqref{item:3_mflds}. 

\begin{remark}
    Given a finitely generated group $G$, the product $G\times \mathbb{Z}$ has an infinite order central element, and hence is a MLTG group. This illustrates the first and second drawback about the MLTG property: it can only inform about the behaviour of Morse quasi-geodesics. If a property does not hold for $(\text{any f. gen. group})\times \mathbb{Z}$, then it cannot be a consequence of MLTG. This is not as bad as it may sound, as taking the direct product with $\mathbb{Z}$ is a procedure that destroys (coarse) negative curvature.
\end{remark}

We postpone the discussion of examples of groups that do not satisfy the MLTG property to the appropriate sections that appear later in the text (see Examples \ref{counterex:1}, \ref{counterex:2}, \ref{counterex:3}).

As a first example of the consequence of the MLTG property we have the following Cartan-Hadamard--type theorem.

\begin{theorem}[{\cite{RussellSprianoTran:thelocal}}]
Let $X$ be a metric space with the MLTG property. For each $\delta \geq 0$, there exists $R \geq 0$ so that if every triangle in $X$ with vertices contained in a ball of radius $R$ is $\delta$--slim, then $X$ is a Gromov hyperbolic space.
\end{theorem}

\subsection{Stability and subgroups} \label{appendix:stability_section}
We begin with a collection of properties of MLTG groups that are more or less straightforward consequences of the definition. The first fact is that in a MLTG group one can promote a not necessarily periodic Morse ray to a periodic, bi-infinite quasi-geodesic. In other words, the existence of a Morse ray implies the existence of a stable, two-ended subgroup.
 
\begin{lemma}[{\cite{RussellSprianoTran:thelocal}}]\label{appendix:ray_to_element}
    Let $G$ be a group containing an infinite quasi-geodesic Morse ray. Then $G$ contains a Morse element. 
\end{lemma}

This allows us to find the first group without the MLTG property. 

\begin{example}\label{counterex:1}
In \cite{fink:morse}, the author shows that there is a torsion group with a Morse ray. As such a group cannot contain a Morse element, it follows that it cannot have the MLTG property.
\end{example}

 Stable subgroups were introduced in \cite{DurhamTaylor:convexcocompact} and have become a very active field of study. Roughly, these are subgroups where all the directions are Morse quasi-geodesics in a uniform way. Studying the behaviour of stable subgroups was the reason that originally motivated the study of MLTG spaces.

\begin{definition}
    Let $G$ be a group with a fixed finitely generated set $S$. A subgroup $H$ is \emph{stable} if there exists a Morse gauge $M$ and a radius $R$ such that in $\mathrm{Cay}(G,S)$ every two points of $H$ can be joined by a $M$--Morse geodesic that is contained in the $R$--neighbourhood of $H$.
\end{definition}
It is not hard to see that the stability of a subgroup is independent of the choice of (finite) generating set, hence we can talk about stable subgroups without specifing the generating set for $G$.
An alternative definition due to Tran (\cite[Theorem~4.8]{Tran:onstrongly}) is that a subgroup of a finitely generated group is stable if and only if it is hyperbolic and \emph{strongly quasi-convex}, in the following sense: $Y\subseteq X$ is strongly quasi-convex if there exists a Morse gauge $M$ so that every $(\la, \ka)$--quasi-geodesic with endpoints on $Y$ is contained in the $M(\la, \ka)$--neighbourhood of $Y$. 

To date, the main theorem about combinations of stable subgroups in MLTG groups is \cite[Theorem~3.1]{RussellSprianoTran:thelocal}, which is a direct extension of Gitik's work on hyperbolic groups to a significantly larger class \cite{Gitik(1999b)}. However, in the hypotheses of both \cite{Gitik(1999b)} and \cite{RussellSprianoTran:thelocal} it is required that a certain group is malnormal. This originates from the fact that Gitik's result relies on \cite[Lemma~8]{Gitik(1999b)}, a result about the intersection of neighbourhoods of cosets of malnormal subgroups. However, \cite[Lemma~8]{Gitik(1999b)} is superseeded by the following more general lemma. 

\begin{lemma}\label{appendix:improved_malnormal}
   Let $H$ be an almost malnormal stable subgroup of a finitely generated group $G$, and let $\delta$ be a non-negative constant. Let $\gamma_1, \gamma_2$ be two left translates of geodesic segments in $G$ representing elements of $H$, with starting points belonging to two different cosets of $H$. Then there exists $R = R(G, H, \delta)$ so that any subpath of $\gamma_2$ which belongs to the $\delta$--neigbourhood of $\gamma_1$ is shorter than $R$.
\end{lemma}
The important part of Lemma~\ref{appendix:improved_malnormal} is that the bound $R$ only depends on $H$, and not on the specific cosets that contain the endpoints of $\gamma_1, \gamma_2$. 
\begin{proof} 
    Let $D$ be the maximal length of a segment of $\gamma_1 \cap N_\delta(\gamma_2)$, and let $H_i$ be the coset of $H$ that contains the endpoints of $\gamma_i$, $i=1,2$. The stability of $H$ implies that there exists $M$ depending only on $H$ so that both segments $\gamma_i$ are $M$--Morse, thus \cite[Proposition~2.4]{Cordes:Morse} gives a constant $N_1=N_1(M)$, so that there is a subsegment of $\gamma_1$ of length $D-4\delta$ that belongs to the $N_1$--neighbourhood of $\gamma_2$. This does not exactly tell us that there are $D-4\delta$ points of $H_1$ in the $N_1$--neighbourhood of $H_2$, as $\gamma_i$ is only in a neighbourhood of $H_i$. However, it tells us that there exists a constant $N_2 = N_2(H)$ so that there are $(D-4\delta)/N_2 -1$ points of $H_1$ that are in the $N_1 + N_2$ neighbourhood of $H_2$. There are only finitely many words of length $N_1 + N_2$. Thus for any $k$ there exists $K$ so that if $D\geq K$, then there exists a word $w$ and at least $k$ points on $H_1$ that can be joined to $H_2$ by a geodesic labeled by $w$. This means $\vert H \cap H^w\vert\geq k$. Since $H$ is almost malnormal and the cosets $H_1$ were chosen to be disjoint, $H \cap H^w$ is a finite group. Since $H$ is stable, it is a hyperbolic group by \cite[Theorem~4.8]{Tran:onstrongly} and thus it has at most finitely many conjugacy classes of finite subgroups (see for instance \cite[Theorem~III.$\Gamma$.3.2]{BridsonHaefliger}). This yields an upper bound on $k$ that only depends on $H$, which in turn yields an upper bound on $D-4\delta$ that only depends on $H$ and $G$, and, finally, this yields an upper bound on $D$ that only depends on $G,H, \delta$.
\end{proof}

\begin{remark}
    As pointed out, Lemma~\ref{appendix:improved_malnormal} allows to relax the assumption of malnormality to almost malnormality in \cite{RussellSprianoTran:thelocal}, Theorem 3.1.2 and Corollary 3.3.2, and in turn it allows to remove the torsion-free assumption from Corollary 3.5.
\end{remark}

The combination theorem for stable subgroups essentially states that under certain conditions the subgroup generated by two stable subgroups is as large as possible. 

This is a condensed version, following the introduction of \cite{RussellSprianoTran:thelocal}. The interested reader can find a more detailed versions in \cite[Section~3.1]{RussellSprianoTran:thelocal}. 
\begin{theorem}[{\cite{RussellSprianoTran:thelocal} and Lemma~\ref{appendix:improved_malnormal}, combination Theorem for stable subgroups}]\label{appendix:combination_stable}
    Let $G$ be a finitely generated group with the Morse local-to-global property and let $P , Q$ be stable subgroups of $G$. For each finite generating set of $G$, there exists $C > 0$ (depending only on the stability parameters of $P$ and $Q$) so that if $P \cap Q$
contains all elements of $P \cup Q$ whose length in $G$ is at most $C$, then the subgroup $\langle P, Q\rangle$ is stable in $G$ and isomorphic to $P \ast_{P\cap Q}Q$. Further, if $P$ is almost malnormal in $G$, then the same conclusion holds if we only require $P \cap Q$ to contain all elements of $P$ whose length in $G$ is at most $C$.
\end{theorem}
\begin{proof}
    In \cite[Theorem~3.1.1]{RussellSprianoTran:thelocal}, malnormality was only used because it matched the hypotheses needed in \cite[Theorem~2]{Gitik(1999b)}. In the latter, malnormality was only used to invoke \cite[Lemma~8]{Gitik(1999b)}, which is superseeded by Lemma~\ref{appendix:improved_malnormal}.
\end{proof}

An interesting consequence of Theorem~\ref{appendix:combination_stable} is the following,  which is a generalization of \cite[Theorem~1]{Arzhantseva:quasiconvex}  to stable subgroups of Morse local-to-global groups.
\begin{corollary}[{\cite{RussellSprianoTran:thelocal} and Lemma~\ref{appendix:improved_malnormal}, abundance of free subgroups}]\label{appendix:MLTG_Morse_imply_free}
    Let $G$ be a Morse local-to-global group. If $Q$ is an infinite, infinite index stable subgroup of $G$, then there is an infinite order element h such that $\langle Q, h\rangle \cong Q  \ast \langle h \rangle$ and $\langle Q, h \rangle$ is stable in $G$.
\end{corollary}
\begin{proof}
    The proof is identical to \cite[Corollary~3.5]{RussellSprianoTran:thelocal}. In that prove, the authors argue that a certain group $P$ is malnormal. However, what is actually proven is that if $uPu^{-1}\cap P$ is infinite, then a contradiction follows. Then they use the torsion-free hypothesis to conclude that $P$ is malnormal. Without the torsion-free hypothesis, this exactly amounts to proving that $P$ is almost malnormal. 
\end{proof}

This allows us to introduce the second class of examples of groups that do not have the MLTG property.
\begin{example}\label{counterex:2}
     Let $G$ be a non-elementary, torsion-free MLTG group with a Morse element $g$. Then $\langle g\rangle$ is stable and setting $Q =\langle g\rangle$ in Corollary~\ref{appendix:MLTG_Morse_imply_free} allows to deduce that $G$ has a non-abelian free subgroup, that is moreover stable.
     
Thus, a non-elementary, torsion-free group containing Morse elements but not stable non-abelian free subgroups is not MLTG.  
     
For instance, in \cite{OOS}, the authors build examples of torsion-free, non-elementary groups all of whose elements are Morse, but so that all of their proper subgroups are cyclic. Hence, those groups do not have the MLTG property. 
\end{example}

Using the combination theorem for stable subgroups \cite{RussellSprianoTran:thelocal} (cf Theorem~\ref{appendix:combination_stable}), one can obtain the following.

\begin{corollary}[{\cite[Corollary~3.6]{RussellSprianoTran:thelocal}}]
    Let $G$ be a finitely generated
group with the Morse local-to-global property. If $G$ contains a Morse element and $N$ is an infinite normal subgroup of $G$, then $N$ contains a Morse element of $G$.
\end{corollary}

\subsection{Algorithmic properties} 
Hyperbolic groups are well-known for their good algorithmic properties, such as the solvability of the word and conjugacy problem, as well as the solvability of the much harder isomorphism problem \cite{Dahmani-Guirardel(2011)}, the fact that their growth series is rational \cite{Cannon(1984)}, and the fact that their geodesics are characterized in terms of regularity of certain languages \cite{hughesnairnespriano:regularity}. 
To corroborate the expectation that \say{in MLTG groups Morse geodesics behave as in the hyperbolic groups}, there are many results about algorithmic properties of Morse quasi-geodesics and stable subgroups for MLTG groups. 

Given $H\leq G$ and a finite generating set $S$ for $G$, we define the \emph{relative growth function} of $H$ with respect to $S$ to be the function $\mathrm{Growth}_{H,S} \colon \mathbb{N} \to \mathbb{N}$ that assigns to each $n$  the number of elements of $H$ whose
distance from the identity in the Cayley graph $\mathrm{Cay}(G,S)$ is at most $n$. The \emph{growth rate} is defined to be 
\[\lambda_{H, S} = \limsup\limits_{n \to \infty} \sqrt[n]{\mathrm{Growth}_{H,S}(n)}.\]

In Morse local-to-global groups there is a strong growth gap for infinite-index stable subgroups. Previous results of this kind were only known for hyperbolic groups, and certain classes of subgroups of relatively hyperbolic and cubulated groups \cite{DahmaniFuterWise:growth}. The Morse local-to-global property provides the correct setting and tools to extend the results to a much larger class. As one of the ingredient of the following theorem is the original version of Corollary~\ref{appendix:MLTG_Morse_imply_free}, we are able to relax the original hypotheses to the following. 
\begin{theorem}[{\cite{CordesRussellSprianoZalloum:regularity}, growth gap for stable subgroups}]\label{appendix:growth_gap}
Let $G$ be a group with a finite generating set $S$ and with the Morse local-to-global property. Let $H < G$ be infinite and of infinite index. Then the growth rate of $H$ with respect to $S$ is strictly smaller than the growth rate of $G$ with respect to $S$, namely:
\[ 
\lambda_{H,S} < \lambda_{G,S}.
\]
\end{theorem}
\begin{proof}
    The proof of \cite[Theorem~5.1]{CordesRussellSprianoZalloum:regularity}
    relies on the fact that if $G$ is torsion-free (or an alternative condition), then for a subgroup $H$ as above there is $h$ so that $\langle H, h \rangle \cong H \ast \langle h \rangle$. This is the only place where the hypothesis is used, as \cite[Theorem~5.1]{CordesRussellSprianoZalloum:regularity} follows from \cite[5.2]{CordesRussellSprianoZalloum:regularity}, that only needs the subgroup as an input. Corollary~\ref{appendix:MLTG_Morse_imply_free} yields the appropriate subgroup without the torsion-free hypothesis that can hence be removed from the above Theorem.
\end{proof}

The proof or Theorem~\ref{appendix:growth_gap} relies on two core results about algorithmical properties of words in Morse local-to-global groups. The first result is a  Morse-geodesic-version of the Gersten and Short result that a subgroup $H$ of a hyperbolic group is quasi-convex if and
only if the language of geodesic words representing $H$ is regular \cite{GerstenShort:rational}. 
\begin{theorem}[{\cite{CordesRussellSprianoZalloum:regularity}, languages for stable subgroups}]\label{appendix:stable_language}
Let $G$ be a Morse local-to-global group generated by a finite set $S$. For a subgroup $H < G$, let $L_H$ be the language of geodesic words in the generating set $S$ representing elements of $H$. Then $H$ is stable if and only if $L_H$ is a regular language and all elements of $L_M$ are $M$--Morse, for some Morse gauge $M$.
\end{theorem}

An important corollary of Theorem~\ref{appendix:stable_language} is the following. 

\begin{corollary}[{\cite{CordesRussellSprianoZalloum:regularity}, rationality of growth series}]\label{appendix:rationality_series}
Let $G$ be a Morse local-to-global group  and let $H$ be a stable subgroup of $G$. Given a finite generating set $S$ of $G$, there exist polynomials $P (x), Q(x) \in  Q[x]$ such that 
\[
\sum_{n=0}^{\infty} \mathrm{Growth}_{H,S}(n)x^n= \frac{P(x)}{Q(x)}.
\]
\end{corollary}

Notably, Corollary~\ref{appendix:rationality_series} applied to convex-cocompact subgroups of the Mapping Class Group answers question \cite[Chapter 2, Question 3.14]{Farb:problems} of Farb and is thus further evidence that the MLTG property is the correct setting to study Morse geodesics in.

The second core technical result is a language-theoretic characterization of Morse geodesics in MLTG groups. Intuitively, one expects that \say{local-to-global} means that one needs to check only finite information to asses whether a geodesic is Morse. Although this turns out to be correct, the proof of the next theorem is surprisingly subtle, and it requires a slightly different stratification of the Morse boundary compared to the usual one. 

\begin{theorem}[{\cite{CordesRussellSprianoZalloum:regularity}, languages for Morse geodesics}]\label{appendix:reg_languages}
    Let $G$ be a finitely generated group with the Morse local-to-global property and let $S$ be a finite generating set. For each Morse gauge
$M$, there exists a regular language $L_M$ such that
\begin{enumerate}
\item every $M$--Morse geodesic word of $G$ is contained in $L_M$;
\item every element of $L_M$ is an $M'$--Morse geodesic word, where $M'$ is determined by $M$.
\end{enumerate}
\end{theorem}

\subsection{Topological properties}
 A group $G$ can be equipped with the profinite topology, namely the topology on $G$ whose basic open subsets are cosets of finite index subgroups of $G$. The notion of separability in the profinite topology is a very well-studied topic. For instance, the trivial subgroup is separable in $G$ if and only if $G$ is residually finite. Moreover, knowing that a certain subgroup is separable gives significant information about it, for instance it allows to solve the membership problem for such a group.  Ribes and
Zalesskii started the line of enquiry of asking whether the product of certain subgroups is separable, focusing on convex subgroups of the free group (\cite{RibesZalesskii:profinite}), a study that was then further developed by Minasyan and Minasyan--Mineh (\cite{Minasyan:separable, minasyanmineh:quasiconvexity}).

The Morse local-to-global property allows to check that stable subgroups are product separable in a very general setting. 
\begin{theorem}[{\cite{minehspriano:separability}}]\label{appendix:product_separable}
Let $G$ be a finitely generated group with the Morse local-to-global property, and suppose that any stable subgroup of $G$ is separable. Then the product of any stable subgroups of $G$ is separable.    
\end{theorem}
The class of groups to which Theorem~\ref{appendix:product_separable} applies is quite large. For instance, if a group is LERF then any stable subgroup of it is separable. Hence we have:

\begin{corollary}[{\cite{minehspriano:separability}}]
 Let G be a finitely generated LERF group with the Morse local-to-global property. Then the product of any stable subgroups of G is separable.
\end{corollary}

Other classes of groups where the result holds are the following. 
\begin{corollary}[{\cite{minehspriano:separability}}]
A product of stable subgroups in a virtually special group is separable. A product of strongly quasi-convex subgroups in a right-angled-Artin group is separable. 
\end{corollary}
Recently, Sam Shepherd showed with different techniques that a product of convex subgroups of a virtually special group is separable (\cite{shepherd:product}).

\subsection{Dynamical properties}

A very fruitful avenue of study for hyperbolic groups is through their boundaries. The Gromov boundary of a hyperbolic group is a compact space on which the group acts by homeomorphisms with nice dynamical properties. A deep theorem of Bowditch \cite{Bowditch:TopChar} shows that this is in fact a characterization of hyperbolicity. Hence, for a non-hyperbolic group it is difficult to find a notion of boundary that yields a  compact space with good dynamical properties for the action of the group. Regarding compactness, the Morse local-to-global property allows to get the next best result. 

\begin{theorem}[{\cite{HeSprianoZbinden:sigmacompactnes}}]\label{appendix:sigma_cpt}
    Let $G$ be a group satisfying the Morse local-to-global property. Then the Morse boundary of $G$ is strongly $\sigma$--compact.
\end{theorem}
Although Theorem~\ref{appendix:sigma_cpt} is concerned with the topology of the boundary, it is worth remarking that the techniques needed to prove it strongly rely on the algorithmical characterization from Theorem~\ref{appendix:reg_languages}.

Theorem~\ref{appendix:sigma_cpt} allows us to present the final class of examples of groups that do not have the Morse local-to-global property. 
\begin{example}\label{counterex:3}
    In \cite{Zbinden:SmallCancellation} the first examples of groups with non $\sigma$--compact Morse boundary were produced, which in particular cannot have the Morse local-to-global property. 
\end{example}

It is worth noting that all groups that are known to not have the Morse local-to-global property are not finitely presented, which motivates the following question. 
\begin{question}
    Is there a finitely presented group that does not have the Morse local-to-global property?
\end{question}

 In light of Theorem~\ref{appendix:sigma_cpt}, the following is a natural weakening of the previous question (See \cite[Question~1]{HeSprianoZbinden:sigmacompactnes}).
\begin{question}
    Let $G$ be a finitely presented group with strongly $\sigma$--compact Morse boundary. Does $G$ have the MLTG property?
\end{question}

It is natural to wonder if, as in the case of hyperbolic groups, one can characterize Morse local-to-global groups by the topology and dynamics on their boundaries. In this direction, there is a partial result for a particular type of infinitely presented groups. 

\begin{theorem}\cite{HeSprianoZbinden:sigmacompactnes}
    Let $G = \langle S\mid R\rangle $ be a $C'(1/9)$ group. Then $G$ has the Morse local-to-global property if and only if the Morse boundary of $G$ is strongly $\sigma$--compact. 
\end{theorem}

As mentioned, Morse local-to-global groups tend to have stable free subgroups (Corollary~\ref{appendix:MLTG_Morse_imply_free}). In addition to algebraic properties, this has deep implications about the possible measures on boundaries associated with Morse local-to-global groups. The next theorem is phrased using the notion of \emph{geodesic boundary} of \cite{HeSprianoZbinden:sigmacompactnes}. Roughly, this means any notion of boundary in which the sentence \say{the points in the boundary coming from Morse geodesics} is meaningful. 

\begin{theorem}[{\cite{HeSprianoZbinden:sigmacompactnes}}]\label{appendix:measures}
Let $G$ be a non-hyperbolic group with the Morse local to-global property and acting geometrically on a space containing a contracting ray. Let $\mu$ be a probability measure on $G$.
Then for any geodesic boundary $B$ of $G$ the image of the set of Morse rays has measure zero in $B$ with
respect to any $\mu$–stationary measure.
\end{theorem}
The main example of a geodesic boundary is the horofunction boundary, which we remark admits a stationary measure since it is compact. Other examples are the HHG boundary of $G$ in the case when $G$ is an HHG, or the quasi-redirecting boundary if $G$ is an ATM space as defined in \cite{qingrafi:quasi}.

We conclude by recording interesting facts about the dynamics of Morse elements in Morse local-to-global groups. 

\begin{theorem}[{\cite{CordesRussellSprianoZalloum:regularity}}]
Let $G$ be a Morse local-to-global group.
\begin{enumerate}
    \item The set of attracting fixed points of Morse elements of G is dense in the Morse boundary $\partial_\ast G$.
    \item If $H$ is an infinite normal subgroup of $G$, then the limit set of $H$ in $\partial_\ast G$ coincides with $\partial_\ast G$.
\end{enumerate}
\end{theorem}

\subsection{Relation to acylindricity and Morse detectability}
We conclude the survey with an overview of what, according to the authors, are some interesting research directions stemming from the Morse local-to-global property. The first one relates to a major open question in the field of quasi-isometric classification of groups. It is an easy exercise to show that if $G\to H$ is an \emph{equivariant} quasi-isometry and $H$ is acylindrically hyperbolic, then so it is $G$. However, whether the property of being acylindrically hyperbolic is quasi-isometric invariant in the class of groups is a longstanding open problem, highlighting the difficulty of deducing algebraic data from geometric ones. A motivating question for the study of Morse local-to-global groups is the following. 

\begin{question}\label{appendix:question_AH}
    Let $G$ be an acylindrically hyperbolic Morse local-to-global group, and let $H$ be quasi-isometric to $G$. Is $H$ acylindrically hyperbolic?
\end{question}
Note that since the MLTG property is preserved by quasi-isometry, $H$ is an MLTG group as well. Thus, the above Question is implied by the following.
\begin{question}\label{appendix:question_Morse_imply_AH}
    Let $G$ be a non virtually-cyclic Morse local-to-global group with non-empty Morse boundary. Is $G$ acylindrically hyperbolic?
\end{question}
Reformulating Question~\ref{appendix:question_AH} to Question~\ref{appendix:question_Morse_imply_AH} already shows one of the main strengths of the MLTG property, namely to allow the promotion of geometric data to algebraic ones using Lemma~\ref{appendix:ray_to_element}. Moreover, the MLTG property seems to imply that the Morse elements are \say{well behaved}. The first attempt to answer Question~\ref{appendix:question_Morse_imply_AH} is to pick a Morse element and use its axis to attempt the projection complex construction of \cite{BBF}, which would yield an acylindrical action on a hyperbolic space. However, the strategy as stated cannot work due to the following Theorem.

\begin{theorem}[{\cite[Theorem~1.1]{abbottzbinden:non-loxodromic}}]\label{appendix:acylindrical-mltg-monster}
    There exists a Morse local-to-global group $G$ and an infinite order Morse element $a\in G$ such that $a$ is not loxodromic in any action of $G$ on a hyperbolic space. In particular, the element $a$ is not WPD in any such action.
\end{theorem}

Although the above Theorem seems discouraging, there are many reasons why it is not. Firstly, such a group $G$ is acylindrically hyperbolic, not contradicting Question~\ref{appendix:question_Morse_imply_AH}. Secondly, the group constructed in Theorem~\ref{appendix:acylindrical-mltg-monster} is not finitely presented. Perhaps this suggests that Question~\ref{appendix:question_Morse_imply_AH} should be asked for finitely presented groups instead.

Another version of the question relates to the notion of \emph{Morse detectability}. A space $X$ is Morse detectable if there exists a hyperbolic space $Y$ and a coarsely Lipschitz projection $\pi \colon X \to Y$ so that a quasi-geodesic $\gamma$ of $X$ is Morse if and only if the composition $\pi\circ\gamma$ is a quasi-geodesic of $Y$, where the Morse gauge and quasi-geodesic constants each depend on the other (\cite[Definition~4.17]{RussellSprianoTran:thelocal}). The main source of inspiration is the Mapping Class Group, where the curve graph is the associated hyperbolic space \cite{MasurMinsky:complex1, DurhamTaylor:convexcocompact}. Any Morse detectable space satisfies the MLTG property, whereas the opposite is not known. This allows us to consider a stronger notion of MLTG groups, and hence we ask the natural following question.
\begin{question} 
    Is the following true: All MLTG groups with equivariant Morse detectability space are acylindrically hyperbolic? 
\end{question}
There has been recent progress on constructing hyperbolic spaces that satisfy Morse detectability properties (\cite{PetytSprianoZalloum:hyperbolic, petytzalloumspriano:constructing, Zbinden:hyperbolic}), making this a very exciting moment for the study of Morse detectable spaces.

\bibliographystyle{alpha} 
\bibliography{cornelia}

\end{document}